\documentclass[11pt, a4paper, english]{smfart}
\usepackage{smfthm}
\usepackage{amssymb}
\usepackage{tikz-cd}
\usepackage{mathtools}
\usepackage{hyperref}
\usepackage{enumitem}
\usepackage{comment}

\renewcommand{\geq}{\geqslant}

\theoremstyle{plain}
\newtheorem{theorem}{Theorem}[section]

\newtheorem{definition}[theorem]{Definition}
\newtheorem{lemma}[theorem]{Lemma}
\newtheorem{conjecture}[theorem]{Conjecture}
\newtheorem{corollary}[theorem]{Corollary}
\newtheorem{proposition}[theorem]{Proposition}

\theoremstyle{remark}
\newtheorem{remark}[theorem]{Remark}
\newtheorem{example}[theorem]{Example}

\renewcommand{\AA}{\ensuremath{\mathbb{A}}}
\newcommand{\CC}{\ensuremath{\mathbb{C}}}

\newcommand{\PP}{\ensuremath{\mathbb{P}}}
\newcommand{\QQ}{\ensuremath{\mathbb{Q}}}

\newcommand{\ZZ}{\ensuremath{\mathbb{Z}}}

\DeclareMathOperator{\GL}{GL}

\DeclareMathOperator{\End}{End}

\DeclareMathOperator{\Hom}{Hom}

\DeclareMathOperator{\MT}{MT}

\DeclareMathOperator{\Aut}{Aut}
\DeclareMathOperator{\CSpin}{CSpin}

\DeclareMathOperator{\h}{\mathsf{h}} 
\DeclareMathOperator{\NS}{\mathrm{NS}} 
\DeclareMathOperator{\Km}{\mathrm{Km}}

\title[K3 surfaces associated with varieties of generalized Kummer type]{K3 surfaces associated with varieties of generalized Kummer type}
\author{Salvatore Floccari}
\address{Institut f\"ur Mathematik, Humboldt-Universit\"at zu Berlin, Germany}
\email{salvatore.floccari@hu-berlin.de}

\begin{document}
	\maketitle

\begin{abstract}
With any hyper-K\"ahler variety $K$ of generalized Kummer type is associated via Hodge theory a K3 surface $S_K$. We show how they are related geometrically through a moduli space of sheaves on $S_K$. As a consequence, building fundamentally on the works of O'Grady, Markman, Voisin, Varesco, we establish the Hodge conjecture for all powers of any of these K3 surfaces as well as for all abelian fourfolds of Weil type with discriminant 1 and their powers, strenghtening a result of Markman.
\end{abstract}

\section{Introduction}
A manifold of generalized Kummer type is by definition a compact K\"ahler manifold which is deformation equivalent to the generalized Kummer variety $K^n(A)$ on an abelian surface $A$ constructed by Beauville in \cite{beauville1983varietes}. Manifolds of generalized Kummer type give examples of hyper-K\"ahler manifolds (\cite{beauville1983varietes, Huy99}) in each even dimension $2n$, commonly referred to as manifolds of $\mathrm{Kum}^n$-type; by a variety of $\mathrm{Kum}^n$-type we shall indicate one such manifold which is projective. 
For any $n>1$, these projective varieties form countably many $4$-dimensional families. Besides Beauville's construction of generalized Kummer varieties, other examples can be obtained by considering moduli of stable sheaves on abelian surfaces \cite{Yos01}; however, it is known that the general projective variety of generalized Kummer type cannot be realized from such a construction. 

In this article we consider some K3 surfaces naturally attached via Hodge theory to varieties of generalized Kummer type. Let $K$ be a variety of $\mathrm{Kum}^n$-type. By the Torelli theorem there exists a unique K3 surface $S_K$ such that there exists a Hodge isometry $H_{\mathrm{tr}}^2(S_K,\ZZ)\xrightarrow{\ \sim \ } H^2_{\mathrm{tr}}(K,\ZZ)(2)$
of transcendental lattices, where on the right-hand side the Beauville-Bogomolov form is multiplied by a factor $2$. If $K=K^n(A)$ is the generalized Kummer variety on an abelian surface $A$, then $S_K$ is the Kummer K3 surface $\Km(A)$, which motivates our definition. 

We will explain how the K3 surface $S_K$ is related geometrically to $K$, for any variety~$K$ of $\mathrm{Kum}^n$-type. Indeed, there exists a natural subvariety $W_K$ of $K$ which is a hyper-K\"ahler variety of $\mathrm{K}3^{[m]}$-type where either $2m=n$ or $2m=n+1$, with the key property of deforming everywhere together with $K$; i.e., $W_K$ is a trianalytic submanifold of $K$ in the sense of Verbitsky~\cite{VerbitskyTrianalytic}.
It is obtained as the component of maximal dimension of the fixed locus of an involution of $K$ which acts trivially on the second cohomology group; such automorphisms were studied already in \cite{hassettTschinkel, boissiere2011higher, oguiso2020no, KMO, floccariKum3} and used to produce trianalytic cycles in~\cite{kaledin1998}. 
Our contribution lies in the determination of the period of $W_K$: we show that the pull-back along $\iota\colon W_K\hookrightarrow K$ induces a primitive embedding of lattices and Hodge structures $\iota^*\colon H^2_{\mathrm{tr}}(K,\ZZ)(2)\xrightarrow{\ \sim \ } H^2_{\mathrm{tr}}(W_K,\ZZ)$, where on the left-hand side the form is multiplied by~$2$. It follows that $W_K\subset K$ is a smooth and projective moduli space of stable sheaves on the K3 surface $S_K$. 

As a consequence of the above construction, we obtain the following result.
\begin{theorem}\label{thm:1.1}
	There exists an algebraic cycle on $S_K\times K$ inducing a Hodge isometry 
	\[H^2_{\mathrm{tr}}(S_K,\QQ)\xrightarrow{ \ \sim \ } H^2_{\mathrm{tr}}(K,\QQ)(2).\]
\end{theorem}
This algebraic cycle gives in fact an isomorphism between the transcendental parts of the homological motives of $S_K$ and $K$ in degree $2$, see Theorem \ref{thm:half2}. 

Building on results of O'Grady \cite{O'G21}, Markman \cite{markman2019monodromy}, Voisin \cite{voisinfootnotes} and Varesco \cite{varesco}, we use Theorem \ref{thm:1.1} to prove some cases of the Hodge conjecture. 
\begin{theorem}\label{thm:2}
	Let $K$ be any projective variety of $\mathrm{Kum}^n$-type, and let $S_K$ be the $\mathrm{K}3$ surface with transcendental lattice $H^2_{\mathrm{tr}}(S_K,\ZZ) \xrightarrow{\sim}  H^2_{\mathrm{tr}}(K,\ZZ)(2)$. Then:
	\begin{enumerate}[label=(\roman*)]
		\item the Kuga-Satake correspondence is algebraic for $S_K$;
		\item the Hodge conjecture holds for any power of $S_K$.
	\end{enumerate}
\end{theorem}
The theorem gives countably many new families of K3 surfaces of general Picard rank $16$ for which it is possible to prove $(i)$ and $(ii)$, generalizing previous results from \cite{paranjape,ILP,floccariKum3}. 
Using that isogenies of K3 surfaces are algebraic by \cite{Buskin, huybrechtsMotives}, we show that the conclusions of Theorem \ref{thm:2} hold for any K3 surface $ S$ such that there exists an isometric embedding of rational quadratic spaces $H^2_{\mathrm{tr}}(S,\QQ)\hookrightarrow (\mathrm{U}^{\oplus 3}_{\QQ}\oplus \langle -m\rangle_{\QQ})$ for some positive integer $m$, where $\mathrm{U}$ denotes a hyperbolic plane; see Theorem \ref{thm:K3surfaces}.

Let us summarize the main points of the proof.
Recall that the Kuga-Satake construction (\cite{KUGA1967, deligne1971conjecture}) produces an abelian variety $\mathrm{KS}(X)$ from a polarized hyper-K\"ahler variety $X$, with an embedding of Hodge structures $H^2(X,\QQ)_{\mathrm{prim}}\hookrightarrow H^2(\mathrm{KS}(S)^2,\QQ)$ called the Kuga-Satake correspondence; according to the Hodge conjecture, this embedding should be induced by an algebraic cycle.

It was shown by O'Grady \cite{O'G21} and Markman \cite{markman2019monodromy} that the Kuga-Satake variety of a $\mathrm{Kum}^n$-variety is isogenous to a power of its third intermediate Jacobian $J^3(K)$, which is an abelian fourfold of Weil type with discriminant $1$. Moreover, varying $K$ in polarized families one obtains all complete up to isogeny families of such abelian fourfolds (see also~\cite{vangeemenSpinor}).
Abelian fourfolds of Weil type are characterized by the presence of exceptional Hodge classes: there is a $2$-dimensional space of Hodge classes in their middle cohomology which, for the very general such fourfold, cannot be represented as intersection of divisors (see \cite{vanGeemen, moonenZahrinHodge}). Markman has proven in \cite{markman2019monodromy} that these Hodge-Weil classes are algebraic for any fourfold of Weil type of discriminant $1$; his results further imply that the Kuga-Satake correspondence is algebraic for any variety of $\mathrm{Kum}^n$-type, as shown by Voisin \cite{voisinfootnotes}. 

Theorem \ref{thm:2} is then obtained as follows. By definition, the Kuga-Satake variety of $S_K$ is isogenous to a power of the Kuga-Satake variety of $K$, and, hence, to a power of a Weil fourfold with discriminant $1$. We deduce Theorem \ref{thm:2}.$(i)$ from Theorem~\ref{thm:1.1} and the algebraicity of the Kuga-Satake correspondence for the $\mathrm{Kum}^n$-variety $K$. Statement~$(ii)$ is a consequence of $(i)$ thanks to the work of Varesco \cite{varesco}.

The Hodge conjecture for a very general fourfold $A$ of Weil type with discriminant~$1$ is proven by Markman in \cite{markman2019monodromy}, which implies the Hodge conjecture also for powers of such~$A$ by \cite{varesco, Milne}. However, the results of \cite{markman2019monodromy} do not directly imply the Hodge conjecture for an arbitrary Weil fourfold with discriminant~$1$. We complete the proof of the Hodge conjecture for all Weil fourfolds with discriminant~$1$, and their powers.
 
\begin{theorem}\label{thm:1}
	Let $A$ be an abelian fourfold of Weil type with discriminant $1$. Then the Hodge conjecture holds for $A$ and all of its powers.
\end{theorem}

To prove this statement, we show in Theorem \ref{prop:KSHC} that if $S$ is any K3 surface such that the Kuga-Satake correspondence is algebraic and the Hodge conjecture holds for any power of~$S$, then the Hodge conjecture holds for any power of the Kuga-Satake variety of $S$ as well. We then obtain Theorem \ref{thm:1} from Theorem~\ref{thm:2} using that any Weil fourfold with discriminant~$1$ appears as isogeny factor of the Kuga-Satake variety of some K3 surface $S_K$ as in Theorem~\ref{thm:2}. 

\subsection*{Aknowledgements} I am grateful to Lie Fu, Bert van Geemen, Ben Moonen, \'Angel R\'ios Ortiz and Charles Vial for useful conversations and to the anonymous referee for his comments. I thank Stefan Schreieder for his support. Funded by the Deutsche Forschungsgemeinschaft (DFG, German Research Foundation) – Project-ID 491392403 – TRR 358.

\section{Motives}

\subsection{}
A neutral Tannakian category $\mathsf{C}$ is a $\QQ$-linear abelian rigid tensor category in which $\End(\mathsf{1})=\QQ$ and which admits an exact faithful tensor functor $\omega\colon \mathsf{C}\to \mathsf{Vect}_{\QQ}$. We refer the reader to \cite{deligneMilne} for more details. 
The group of tensor automorphisms of $\omega$ is a (pro-)algebraic group~$\mathcal{G}(\mathsf{C})$ with the property that~$\mathsf{C}$ is $\otimes$-equivalent to the category of finite-dimensional representations of $\mathcal{G}(\mathsf{C})$.
Thus, neutral Tannakian subcategories of~$\mathsf{C}$ correspond bijectively to quotients of $\mathcal{G}(\mathsf{C})$.
Given $M\in\mathsf{C}$, we denote by $\langle M\rangle_{\mathsf{C}}$ the neutral Tannakian subcategory of $\mathsf{C}$ generated by~$M$; more directly, it is defined as the smallest thick and full subcategory containing $M$ and closed under direct sums, duals, tensor products and subobjects. Then $\langle M\rangle_{\mathsf{C}}$ is equivalent to the category of finite-dimensional representations of a linear algebraic group $\mathrm{G}_{\mathsf{C}}(M) \subset \GL(\omega(M))$,
called the Tannakian fundamental group of $M$.

\subsection{}
The category $\mathsf{HS}$ of polarizable pure $\QQ$-Hodge structures is an abelian and semi-simple rigid tensor category, which is neutral Tannakian via the forgetful functor to $\QQ$-vector spaces.
The Tannakian fundamental group of a polarizable Hodge structure $V$ is a reductive and connected algebraic group over $\QQ$ (\cite{moonen1999notes}) denoted by $\MT(V) \subset \GL(V)$ and called the Mumford-Tate group of $V$. It is equivalently defined as the subgroup of $\GL(V)$ which fixes all Hodge classes in any tensor construction
\[
V^{ \otimes i} \otimes V^{\vee, \otimes j} \otimes \QQ(k),
\]
where $\QQ(k)$ is the Tate Hodge structure of weight $-2k$. Provided that $V$ is not entirely of weight~$0$, there is a central torus $\mathbb{G}_{m,\QQ} \subset\MT(V)$ whose action induces the weight decomposition of the objects in $\langle V \rangle_{\mathsf{HS}}$.

\begin{remark} \label{rmk:evenPart-Hodge}
Let $V\in\mathsf{HS}$ be a Hodge structure of odd weight. We will denote by $\langle V \rangle_{\mathsf{HS}}^{\mathrm{ev}}$ the full subcategory of $\langle V\rangle_{\mathsf{HS}}$ of Hodge structures of even weight, which is a neutral Tannakian subcategory of $\mathsf{HS}$.
If $w\colon \mathbb{G}_{m,\QQ}\to \MT(V)$ denotes the inclusion of the weight torus, the objects of $\langle V\rangle^{\mathrm{ev}}_{\mathsf{HS}}$ are precisely the Hodge structures in $\langle V\rangle_{\mathsf{HS}}$ on which $w(-1)$ acts trivially. Hence, we have $\MT(\langle V\rangle_{\mathsf{HS}}^{\mathrm{ev}})=\MT(V)/w(-1)$. It follows that if $W$ is a Hodge structure of even weight in $\langle V\rangle_{\mathsf{HS}}$ such that the induced surjective morphism $\phi\colon \MT(V)\to \MT(W)$ of Mumford-Tate groups is an isogeny of degree $2$, we then have the equality $\langle V\rangle_{\mathsf{HS}}^{\mathrm{ev}}=\langle W\rangle_{\mathsf{HS}}$ of subcategories of $\mathsf{HS}$.
Indeed, since $W$ has even weight, it belongs to $\langle V\rangle^{\mathrm{ev}}_{\mathsf{HS}}$, and, therefore, $\phi$ factors as the composition of the projection $\psi\colon \MT(V)\to \MT(\langle V\rangle^{\mathrm{ev}}_{\mathsf{HS}})$ followed by a surjective morphism $\phi'\colon \MT(\langle V\rangle^{\mathrm{ev}}_{\mathsf{HS}})\to \MT(W)$ of algebraic groups. But both $\phi$ and $\psi$ are isogenies of degree $2$, and, hence, $\phi'$ is an isomorphism; it follows that $\langle W\rangle_{\mathsf{HS}}=\langle V\rangle_{\mathsf{HS}}^{\mathrm{ev}}$.
\end{remark}

\subsection{}
An algebraic class on a smooth and projective complex variety $X$ is a cohomology class $\alpha\in H^{2k}(X,\QQ)$ which is $\QQ$-linear combination of fundamental classes of subvarieties of~$X$. Any algebraic class is a Hodge class, and the Hodge conjecture predicts that, conversely, all Hodge classes are algebraic. 
A correspondence from $X$ to $ Y$ of degree $k$ is by definition an algebraic class in $H^{2k+2\dim X} (X\times Y,\QQ)$; correspondences may be composed (\cite{fulton}). 
Grothendieck's theory of motives is very useful in the study of algebraic cycles; in brief, the main insight is to use correspondences as morphisms. 
We denote by $\mathsf{Mot}$ the category of homological motives over $\CC$ with rational coefficients, see \cite{andre}. Its objects are triples $(X, p , n)$ where $X\in \mathrm{SmProj}_\CC$, $p\in H^{2\dim X} (X\times X,\QQ)$ is an idempotent correspondence, and $n$ is an integer. 
Morphisms from $(X,p,n)$ to $(Y,q,m)$ are defined to be the correspondences $f\in H^{2m-2n+2\dim X}(X\times Y,\QQ)$ of degree $m-n$ such that $f\circ p = q\circ f$.
There is a realization functor $R\colon \mathsf{Mot} \to \mathsf{HS}$ which sends $(X, p, n)$ to $p_* (H^{\bullet}(X,\QQ)(n))$.
Moreover, there is a contravariant functor $\h\colon \mathrm{SmProj}_{\CC} \to \mathsf{Mot}$ which attaches to the smooth and projective variety $X$ its motive $ \h(X)\coloneqq (X,\Delta_X,0)$, where $\Delta_X\in H^{2\dim X} (X\times X,\QQ)$ is the class of the diagonal. The composition $R\circ \h$ is the functor $H^{\bullet}\colon \mathrm{SmProj}_\CC\to \mathsf{HS}$ which attaches to the smooth and projective complex variety $X$ its rational Hodge structure $H^{\bullet}(X,\QQ)$.

\subsection{} 
The category $\mathsf{Mot}$ of homological motives is a pseudo-abelian tensor category via the product of varieties. However, it is not known to be abelian nor Tannakian: these properties are conditional to the validity of Grothendieck's standard conjectures~\cite{grothendieck1969standard}. 
By the work of Jannsen~\cite{Jan92} and Andr\'e \cite{andre1996Motives}, we have the following. 
Given $\mathsf{m}\in\mathsf{Mot}$, we let $\langle \mathsf{m} \rangle_{\mathsf{Mot}}$ be the smallest thick and full subcategory of $\mathsf{Mot}$ containing $\mathsf{m}$ and closed under direct sums, duals, tensor products and subquotients.

\begin{theorem}[{\cite[Theorem 4.1]{Ara06}}] \label{thm:standardConj}
	Let $X$ be a smooth and projective variety. The standard conjectures hold for $X$ if and only if $\langle \h(X) \rangle_{\mathsf{Mot}}$ is a semisimple abelian neutral Tannakian category.
\end{theorem}
Thus, provided that the standard conjectures hold for $X$, we can attach to it a reductive algebraic group $\mathrm{G}_{\mathsf{Mot}}(\h(X))$ over $\QQ$, whose representation category is equivalent to $\langle \h(X) \rangle_{\mathsf{Mot}}$. The algebraic group $\mathrm{G}_{\mathsf{Mot}}(\h(X))$ is characterized as the subgroup of $ \prod_j \GL(H^{j}(X,\QQ))$ fixing exactly the algebraic classes in $H^{2j}(X^k,\QQ(j))$, for any $j$ and~$k$. There is a canonical inclusion $i_X\colon \MT(H^{\bullet}(X,\QQ))\hookrightarrow \mathrm{G}_{\mathsf{Mot}}(\h(X))$, which is an isomorphism if and only if the Hodge conjecture holds for $X$ and all of its powers. 
Moreover, any object of $\langle \h(X) \rangle_{\mathsf{Mot}}$ admits a weight decomposition induced by the action of the weight torus $\mathbb{G}_{m}\subset \MT(H^{\bullet}(X,\QQ))\subset  \mathrm{G}_{\mathsf{Mot}}(X)$. In particular, we have the K\"unneth decomposition $\h(X)=\bigoplus_i \h^i(X)$ in $\mathsf{Mot}$.

\begin{remark}\label{rmk:transcendentalPart}
	Let $X$ be a smooth and projective variety and assume that the K\"unneth component $\h^2(X)$ of $\h(X)$ is well-defined. Thanks to Lefschetz theorem on divisors, this motive splits as $\h^2(X)=\h^2_{\mathrm{tr}}(X)\oplus \h^2_{\mathrm{alg}}(X)$: the realization of $\h^2_{\mathrm{alg}}(X)$ is the N\'eron-Severi group $\NS(X)_{\QQ}$ while that of $\h^2_{\mathrm{tr}}(X)$ is the transcendental part $H^2_{\mathrm{tr}}(X,\QQ)$ of $H^2(X,\QQ)$, the smallest sub-Hodge structure whose complexification contains $H^{2,0}(X)$.
\end{remark}

\begin{remark} \label{rmk:conservative}
	In light of Theorem \ref{thm:standardConj} we shall say that the standard conjectures hold for a motive $\mathsf{m}\in \mathsf{Mot}$ if the category $\langle \mathsf{m}\rangle_{\mathsf{Mot}}$ is an abelian and semisimple neutral Tannakian category, thus equivalent to the category of representations of a reductive algebraic group $\mathrm{G}_{\mathsf{Mot}}(\mathsf{m})$.
	When this happens, the realization functor $R\colon \langle \mathsf{m}\rangle_{\mathsf{Mot}} \to \mathsf{HS}$ is conservative (see \cite[Corollaire 5.1.3.3]{andre}), i.e., a morphism in  $\langle \mathsf{m}\rangle_{\mathsf{Mot}}$ is an isomorphism if and only if its realization is an isomorphism of $\QQ$-Hodge structures.
\end{remark}
\begin{remark}\label{rmk:evenPart-Mot}
	Let $\mathsf{m}\in \mathsf{Mot}$ be a motive for which the standard conjectures hold. The torus $w\colon \mathbb{G}_{m,\QQ}\hookrightarrow \mathrm{G}_{\mathsf{Mot}}(\mathsf{m})$ inducing the weight decomposition for the objects in $\langle \mathsf{m}\rangle_{\mathsf{Mot}}$ coincides with the weight torus $\mathbb{G}_{m,\QQ}\subset \mathrm{MT}(R(\mathsf{m}))$ from Hodge theory. 
	Assume that $\mathsf{m}$ is a motive of odd weight and let $\langle \mathsf{m}\rangle_{\mathsf{Mot}}^{\mathrm{ev}}$ be the full subcategory of $\langle \mathsf{m}\rangle_{\mathsf{Mot}}$ of objects of even weight, which is again abelian, semisimple and neutral Tannakian. Analougously to Remark~\ref{rmk:evenPart-Hodge}, the motives in $\langle \mathsf{m}\rangle_{\mathsf{Mot}}^{\mathrm{ev}}$ correspond to the representations of $\mathrm{G}_{\mathsf{Mot}}(\mathsf{m})$ on which $w(-1)$ acts trivially. Therefore, the natural morphism of algebraic groups $\mathrm{G}_{\mathsf{Mot}}(\mathsf{m})\to \mathrm{G}_{\mathsf{Mot}}(\langle \mathsf{m}\rangle_{\mathsf{Mot}}^{\mathrm{ev}})$ is an isogeny of degree~$2$. Moreover, by the same argument given in Remark \ref{rmk:evenPart-Hodge}, if $\mathsf{n}\in \langle \mathsf{m}\rangle_{\mathsf{Mot}}$ is a motive of even weight such that the corresponding morphism $\mathrm{G}_{\mathsf{Mot}}(\mathsf{m})\to \mathrm{G}_{\mathsf{Mot}}(\mathsf{n})$ is an isogeny of degree $2$, we then have the equality $\langle \mathsf{n}\rangle_{\mathsf{Mot}}=\langle \mathsf{m}\rangle_{\mathsf{Mot}}^{\mathrm{ev}}$ of subcategories of $\mathsf{Mot}$.
\end{remark}

\subsection{}\label{subsec:standardConj}
	The standard conjectures are known to hold for curves, surfaces and abelian varieties; they hold for varieties $X$ and $Y$ if and only if they hold for $X\times Y$, see~\cite{kleiman}.
	The standard conjectures are not known to hold in general for hyper-K\"ahler varieties. They hold for the varieties of $\mathrm{K}3^{[n]}$, $\mathrm{OG}10$ and $\mathrm{OG}6$-type built via moduli spaces of sheaves on K3 or abelian surfaces by \cite{Bue18, FFZ, floccari22}, and for some hyper-K\"ahler varieties admitting a Lagrangian fibration \cite{Ancona2023}. 
	The standard conjectures have been fully proven for varieties of $\mathrm{K}3^{[n]}$-type by Charles and Markman \cite{CM13}. Recently, Foster \cite{foster} proved that they hold in certain degrees for varieties of $\mathrm{Kum}^n$-type; for later reference we state the following consequence of his results (see also \cite[Remark 3.2]{floccariVaresco}). 
	\begin{theorem}[\cite{foster}] \label{thm:foster}
		Let $K$ be a variety of $\mathrm{Kum}^n$-type. The K\"unneth projector $H^{\bullet}(K,\QQ)\to H^2(K,\QQ)$ is algebraic. Moreover, the standard conjectures hold for the K\"unneth component $\h^2(K)\in \mathsf{Mot}$ of the motive of $K$.
	\end{theorem}

\section{The Kuga-Satake construction}

\subsection{}\label{subsec:Clifford} Let $V$ be a $\QQ$-vector space equipped with a non-degenerate symmetric bilinear form~$q$. The Clifford algebra
$$ C(V) \coloneqq \bigoplus_{j\geq 0} V^{\otimes j}/\langle v\cdot v - q(v,v)\rangle_{v\in V}$$
is naturally $\ZZ/2\ZZ$-graded: $C(V) = C^+(V)\oplus C^{-}(V)$.
The Clifford algebra acts on itself via left and right multiplication.
We fix a vector $v_0\in V$ such that $q(v_0,v_0)\neq 0$. Setting $\mu(v)(a)\coloneqq v\cdot a\cdot v_0$ we obtain an embedding
\begin{equation}\label{eq:mu}
	\mu\colon V\hookrightarrow \mathrm{End}(C^+(v)).
\end{equation}
\begin{remark}\label{rmk:left}
	We denote by $\End_{C^+}(C^+(V))$ the algebra of linear endomorphisms of $C^+(V)$ commuting with the action of $C^+(V)$ on itself by right multiplication. 
	There is a canonical isomorphism $C^+(V)=\mathrm{End}_{C^+}(C^+(V))$, i.e., any $f\in \End_{C^+}(C^+(V))$ is induced by left multiplication by some $\alpha_f\in C^+(V)$; see \cite[\S3]{deligne1971conjecture} or \cite[\S4.2]{Andre1996}.
\end{remark}

\subsection{} Assume now that $(V,q)$ is a polarized $\QQ$-Hodge structure of K3-type, by which we mean that $V$ is an effective Hodge structure of weight $2$ with Hodge numbers~$(1,k,1)$. Following \cite{KUGA1967} and \cite{deligne1971conjecture}, the Hodge structure on $V$ determines an effective polarizable Hodge structure of weight $1$ on $C^+(V)$. This thus defines an abelian variety $\mathrm{KS}(V)$ up to isogeny, such that $H^1(\mathrm{KS}(V),\QQ)\cong C^+(V)$. 
With a suitable Tate twist to make~$V$ of weight $0$, the embedding $\mu\colon V(1)\to \mathrm{End}(H^1(\mathrm{KS}(V),\QQ))$ of \eqref{eq:mu} is a morphism of Hodge structures (for any choice of the vector $v_0\in V$). 

\begin{remark}\label{rmk:DeligneExplains}
	It is explained by Deligne \cite[\S4]{deligne1971conjecture} how the Hodge structure on $C^+(V)$ is induced through the action of the Clifford group $\CSpin(V)$, which is a subgroup of $C^+(V)^*$ acting on $C^+(V)$ via left multiplication. This means that the Mumford-Tate group of $H^1(\mathrm{KS}(V),\QQ)$ is contained in $\CSpin(V)\subset \GL(C^+(V))$. Moreover the morphism $\MT(H^1(\mathrm{KS}(V)),\QQ)\to \MT(V)$ induced by the embedding $\mu$ is an isogeny of degree $2$, being the restriction of the double cover $\CSpin(V)\to \mathrm{GO}(V)$, where $\mathrm{GO}(V)$ is the group of linear automorphisms of $V$ preserving the form up to scalar. 
\end{remark}

\subsection{} Let $S$ be a polarized K3 surface or hyper-K\"ahler variety, and consider the polarized Hodge structure of K3-type $V=H^2(S,\QQ)_{\mathrm{prim}}$, the orthogonal to the given polarization, equipped with the restriction $q$ of the form. We may also take for $q$ the restriction of the Beauville-Bogomolov form, a natural pairing on the second cohomology of a hyper-K\"ahler manifold (\cite{beauville1983varietes}).
We define the Kuga-Satake variety $\mathrm{KS}(S)$ of $S$ as the abelian variety obtained from $(V,q)$ via the Kuga-Satake construction. As above, we have an embedding of Hodge structures
\[
\mu\colon H^2(S,\QQ)_{\mathrm{prim}} (1)\hookrightarrow \mathrm{End}(H^1(\mathrm{KS}(S),\QQ));
\]
identifying $H^1(\mathrm{KS}(S),\QQ)^{\vee}$ with $H^1(\mathrm{KS}(S),\QQ)(1)$ by means of a polarization, we get the embedding of Hodge structures
\begin{equation} \label{eq:mu'}
\mu'\colon H^2(S,\QQ)_{\mathrm{prim}}\hookrightarrow H^1(\mathrm{KS}(S),\QQ)^{\otimes 2} \subset H^2(\mathrm{KS}(S)^2,\QQ).
\end{equation}
\begin{conjecture}\label{conj:KSHC}
	The morphism $\mu'$ is induced by an algebraic cycle on $S\times \mathrm{KS}(S)^2$.
\end{conjecture}

This is the Kuga-Satake Hodge conjecture, which is of course a special case of the Hodge conjecture. It has been proven for many K3 surfaces of Picard rank at least~$17$ (\cite{morrison1985}), but it is widely open otherwise. In \cite{paranjape} and \cite{ILP} the conjecture is proven for two $4$-dimensional families of K3 surfaces of general Picard rank $16$; countably many more such families of K3 surfaces for which the conjecture holds are found in \cite{floccariKum3}. The paper~\cite{bolognesi} proves Conjecture \ref{conj:KSHC} for a $9$-dimensional family of K3 surfaces related to certain cubic fourfolds.
The Kuga-Satake Hodge conjecture has been proven for all hyper-K\"ahler varieties of generalized Kummer type by Markman \cite{markman2019monodromy} and Voisin \cite{voisinfootnotes}.

\begin{remark}\label{rmk:smallerKS}
	As a variant, we may define $\mathrm{KS}'(S)$ as the Kuga-Satake variety built from the smaller Hodge structure of K3-type $H^2_{\mathrm{tr}}(S,\QQ)$; then the Kuga-Satake variety $\mathrm{KS}(S)$ built from $H^2(S,\QQ)_{\mathrm{prim}}$ is isogenous to a power of $\mathrm{KS}'(S)$ (see \cite[Chapter~4, Example 2.4]{huyK3}). Moreover, Conjecture~\ref{conj:KSHC} is equivalent to the statement that the embedding $\mu'\colon H_{\mathrm{tr}}^2(S,\QQ)\hookrightarrow H^1(\mathrm{KS}'(S),\QQ)^{\otimes 2} \subset H^2(\mathrm{KS}'(S)^2,\QQ) $ analogous to~\eqref{eq:mu'} is induced by an algebraic cycle.
\end{remark}

\subsection{} We are now ready to prove the main result of this section.
\begin{theorem}\label{prop:KSHC}
	Let $S$ be a projective $\mathrm{K}3$ surface. Assume that the Kuga-Satake Hodge conjecture holds for $S$. Then the Hodge conjecture holds for all powers of $S$ if and only if it holds for all powers of its Kuga-Satake variety $\mathrm{KS}(S)$.
\end{theorem}
We will deduce this result from the following proposition. Recall (\S\ref{subsec:standardConj}) that the standard conjectures hold for $S$ as well as for $\mathrm{KS}(S)$.

\begin{proposition}\label{prop:keyStep}
	Let $S$ be a projective $\mathrm{K}3$ surface. Assume that the Kuga-Satake Hodge conjecture holds for $S$. Then, the motive $\h^2(S)$ belongs to the subcategory $\langle \h^1(\mathrm{KS}(S))\rangle_{\mathsf{Mot}}$ of the category $\mathsf{Mot}$ of homological motives and the induced morphism of algebraic groups $\mathrm{G}_{\mathsf{Mot}}(\h^1(\mathrm{KS}(S)))\to \mathrm{G}_{\mathsf{Mot}}(\h^2(S))$ is an isogeny of degree $2$.
\end{proposition}

\begin{proof}
	By Theorem \ref{thm:standardConj}, the subcategory $\langle \h(\mathrm{KS}(S) \times S)\rangle_{\mathsf{Mot}}$ of $\mathsf{Mot}$ is an abelian and semisimple neutral Tannakian category.
	Hence, if Conjecture~\ref{conj:KSHC} holds for $S$, then the motive $\h^2(S)$ is a direct summand of $\h^2(\mathrm{KS}(S)^2)$, and so it belongs to $ \langle \h^1(\mathrm{KS}(S))\rangle_{\mathsf{Mot}}$. We obtain the commutative diagram
	\[
	\begin{tikzcd}
		\mathrm{G}_{\mathsf{Mot}}(\h^1(\mathrm{KS}(S))) \arrow[two heads]{r}{\psi \ } & \mathrm{G}_{\mathsf{Mot}}(\h^2(S))\\
		\MT(H^1(\mathrm{KS}(S),\QQ)) \arrow[hook]{u}{i_{\mathrm{KS}}} \arrow[two heads]{r}{{\phi}} & \MT(H^2(S,\QQ)) \arrow[hook]{u}{i_S}
	\end{tikzcd}
	\]
	We know from Remark \ref{rmk:DeligneExplains} that the morphism $\phi$ is an isogeny of degree $2$; we will show that $\ker(\psi) = \iota_{\mathrm{KS}}(\ker(\phi))$, which implies that $\psi $ is also an isogeny of degree $2$.
	
	To ease notation, we let $A\coloneqq \mathrm{KS}(S)$ and $V\coloneqq H^2(S,\QQ)_{\mathrm{prim}}$; we denote by $q$ the polarization form on $V$. 
	By Remark \ref{rmk:DeligneExplains}, the action of $C^+(V)$ on $H^1(A,\QQ)$ via right multiplication commutes with $\MT(H^1(A,\QQ))$, i.e., $C^+(V)$ acts via Hodge endomorphisms. By Lefschetz (1,1)-theorem, any such endomorphism is algebraic. Therefore, $\mathrm{G}_{\mathsf{Mot}}(\h^1(A))\subset \GL(H^1(A,\QQ))$ commutes with the $C^+(V)$-action given by right multiplication. By Remark \ref{rmk:left}, we can thus identify $\mathrm{G}_{\mathsf{Mot}}(\h^1(A))$ with a subgroup of $C^+(V)^*$, the group of units in the Clifford algebra, acting on $C^+(V)$ via left multiplication.
	
	Let $g\in \ker\bigl(\psi\colon \mathrm{G}_{\mathsf{Mot}}(\h^1(A))\to \mathrm{G}_{\mathsf{Mot}}(\h^2(S))\bigr)$; then $g\colon H^1(A,\QQ)\to H^1(A,\QQ)$ is left multiplication by some $\alpha_g\in C^+(V)^*$. We claim that $\alpha_g$ is a central element of $C^+(V)$. 
	
	By assumption, the embedding $\mu\colon V(1)\hookrightarrow \mathrm{End}(H^1(A,\QQ))$ is induced by an algebraic cycle. As $g\in \ker(\psi)$, it follows that $g$ commutes with the image of $\mu$. Recall (\S\ref{subsec:Clifford}) that the definition of $\mu$ depends on a vector $v_0\in V$ with $q(v_0,v_0)\neq 0$, such that  $\mu(v)(\beta)\coloneqq v\cdot \beta \cdot v_0$. In the Clifford algebra we have $v_0^{2k}=q(v_0,v_0)^k\cdot e$ where $e $ is the identity element, and any $\beta\in C^+(V)$ is a linear combination of products $v_1\cdots v_{2k}$ of an even number of vectors in~$V$. As $g$ commutes with $\mu(v)$ for any $v\in V$, we find:
	\begin{align*}
		\alpha_g (v_1v_2\cdots v_{2k}) &
		 = \frac{1}{q(v_0,v_0)^{k}} \cdot \alpha_g \cdot (v_1v_2\cdots v_{2k}) \cdot v_0^{2k} \\
		& = \frac{1}{q(v_0,v_0)^{k}} \cdot \bigl(g\circ \mu(v_1) \circ \mu(v_2) \circ \ldots \circ \mu(v_{2k})\bigr)(e) \\
		& = \frac{1}{q(v_0,v_0)^{k}} \cdot \bigl(\mu(v_1) \circ \mu(v_2) \circ \ldots \circ \mu(v_{2k}) \circ g\bigr)(e) = (v_1v_2\cdots v_{2k}) \alpha_g.
	\end{align*}	
	Therefore, $\ker(\psi)$ is contained in the center of $C^+(V)$. But $V=H^2(S,\QQ)_{\mathrm{prim}}$ has odd dimension $21$, and in such case $C^+(V)$ is a central simple algebra (\cite[p.~216]{deligne1971conjecture}).
	The units in the center of $C^{+}(V)$ thus form a torus $\mathbb{G}_{m,\QQ}$ acting on $C^+(V)$ via scalar multiplication; this torus induces the weight filtration on objects of $\langle H^1(A,\QQ)\rangle_{\mathsf{HS}}$ and so it is contained in $\MT(H^1(A,\QQ))$. Hence, $\ker(\psi)$ is contained in $i_{\mathrm{KS}}(\ker(\phi))$. 
\end{proof}

\begin{proof}[Proof of Theorem \ref{prop:KSHC}]
	One direction is immediate: as Conjecture \ref{conj:KSHC} holds for $S$ by assumption, we have $\h^2(S)\in \langle \h^1(\mathrm{KS}(S))\rangle_{\mathsf{Mot}}$, and the Hodge conjecture for all powers of $\mathrm{KS}(S)$ clearly implies the Hodge conjecture for all powers of $S$. 
	
	Let us prove the converse implication. The Hodge conjecture for all powers of $\mathrm{KS}(S)$ predicts that the realization functor $R\colon \mathsf{Mot}\to \mathsf{HS}$ induces a surjection
	\begin{equation}\label{eq:intermediate} \Hom_{\mathsf{Mot}}(\mathsf{Q}(-j), \h^{2j}(\mathrm{KS}(S)^k)) \to \Hom_{\mathsf{HS}}(\mathbb{Q}(-j), H^{2j}(\mathrm{KS}(S)^k,\QQ)),
	\end{equation} 
	for any $k$ and $j$. 
	Only motives and Hodge structures of even weight are involved in~\eqref{eq:intermediate}, and so these $\mathrm{Hom}$-spaces are actually taken in the categories $\langle \h^1(\mathrm{KS}(S))\rangle_{\mathsf{Mot}}^{\mathrm{ev}}$ and $\langle H^1(\mathrm{KS}(S),\QQ)\rangle_{\mathsf{HS}}^{\mathrm{ev}}$, respectively. 
	Proposition \ref{prop:keyStep} and Remark \ref{rmk:evenPart-Mot} yield the equality $\langle \h^1(\mathrm{KS}(S))\rangle_{\mathsf{Mot}}^{\mathrm{ev}} = \langle \h^2(S)\rangle_{\mathsf{Mot}}$ of subcategories of $\mathsf{Mot}$; in particular, $\h^{2j}(\mathrm{KS}(S)^k)$ belongs to $\langle \h^2(S)\rangle_{\mathsf{Mot}}$ for any $j$ and $k$. Therefore, the map in~\eqref{eq:intermediate} is the map 
	\[
	\Hom_{\langle \h^2(S)\rangle_{\mathsf{Mot}}}(\mathsf{Q}(-j), \h^{2j}(\mathrm{KS}(S)^k)) \to \Hom_{\langle H^2(S,\QQ)\rangle_{\mathsf{HS}}}(\mathbb{Q}(-j), H^{2j}(\mathrm{KS}(S)^k,\QQ))
	\]
	given by the realization functor $R$. If the Hodge conjecture holds for all powers of $S$, then $R\colon \langle \h^2(S)\rangle_{\mathsf{Mot}}\to \langle H^2(S,\QQ)\rangle_{\mathsf{HS}}$ is full, and the maps in \eqref{eq:intermediate} are surjective. 
\end{proof}

\subsection{} \label{subsec:functoriality}
For later use, we recall the following functoriality property of the Kuga-Satake construction, following Varesco \cite{varesco2023hodge}.
Let $X$ and $Y$ be hyper-K\"ahler varieties, not necessarily deformation equivalent and possibly of different dimensions. Assume that $\phi\colon H^2_{\mathrm{tr}}(X,\QQ)\to H^2_{\mathrm{tr}}(Y,\QQ)$ is a rational Hodge similitude, i.e., $\phi$ is an isomorphism of $\QQ$-Hodge structures which multiplies the Beauville-Bogomolov form on the left-hand side by some non-zero $k\in \QQ$. As in Remark \ref{rmk:smallerKS}, let $\mathrm{KS}'(X)$ and $\mathrm{KS}'(Y)$ be the Kuga-Satake varieties constructed from $H^2_{\mathrm{tr}}(X,\QQ)$ and $H^2_{\mathrm{tr}}(Y,\QQ)$, respectively. It is then proven in \cite[Proposition 3.1]{varesco2023hodge} that there exists an isogeny $\Psi\colon \mathrm{KS}'(X)\to \mathrm{KS}'(Y)$ of abelian varieties such that the following diagram 
\[
\begin{tikzcd}
	H^2_{\mathrm{tr}}(X,\QQ) \arrow[hook]{d}{\mu'_X} \arrow{rr}{\phi} && H^2_{\mathrm{tr}}(Y,\QQ) \arrow[hook]{d}{\mu_Y'}\\
	H^1(\mathrm{KS}'(X),\QQ)^{\otimes 2} \arrow{rr}{\Psi_*} && H^1(\mathrm{KS}'(Y),\QQ)^{\otimes 2}
\end{tikzcd}
\]
commutes, where $\mu'_X$ and $\mu_Y'$ are the respective Kuga-Satake correspondences and the bottom arrow is the isomorphism induced by the isogeny $\Psi$.
\begin{remark}\label{rmk:functoriality}
	With notation as above, assume that $\phi\colon H^2_{\mathrm{tr}}(X,\QQ)\to H^2_{\mathrm{tr}}(Y,\QQ)$ is induced by an algebraic cycle on $X\times Y$. Then, Conjecture \ref{conj:KSHC} for $Y$ implies Conjecture~\ref{conj:KSHC} for $X$. Indeed, the above diagram yields $\mu_X'=(\Psi_*)^{-1} \circ \mu'_Y\circ  \phi$; note that $(\Psi_*)^{-1}\colon H^1(\mathrm{KS}'(Y),\QQ)^{\otimes 2} \to H^1(\mathrm{KS}'(X),\QQ)^{\otimes 2}$ is induced by an algebraic cycle, so that the Kuga-Satake correspondence $\mu'_X$ is algebraic if $\mu'_Y$ is so.
\end{remark}

\section{Some $\mathrm{K}3^{[m]}$-type submanifolds of generalized Kummer varieties}

\subsection{} Let $A$ be an abelian surface. We denote by $K^n(A)$ the generalized Kummer variety on $A$ of dimension $2n$ introduced in \cite{beauville1983varietes}. It is a hyper-K\"ahler variety constructed as follows: consider the Hilbert scheme $A^{[n+1]}$ of $0$-dimensional subschemes of $A$ of lenght $n+1$, which is a crepant resolution $\nu\colon A^{[n+1]}\to A^{(n+1)}$ of the symmetric power. Denoting by $\Sigma\colon A^{(n+1)}\to A$ the summation map $\Sigma(a_1,\dots,a_{n+1})=\sum_{i=1}^{n+1}a_i$, the variety $K^n(A)$ is defined as the fibre of $\Sigma\circ \nu$ over $0\in A$. 
We let $A^{(n+1)}_0\subset A^{(n+1)}$ be the fibre of $\Sigma$ over $0$; the Hilbert-Chow morphism $\nu$ restricts to a crepant resolution $\nu\colon K^{n}(A)\to A^{(n+1)}_0$. For $n=1$, the construction yields nothing but the Kummer K3 surface $\Km(A)$ associated to $A$, that is, the minimal resolution of the quotient $A/\pm 1$.

We will study certain natural subvarieties of $K^n(A)$, $n\geq 2$. Consider the morphisms:
\begin{equation}
	\begin{split}
f & \colon A^{m}\to A_0^{(2m)\phantom{+1}}, \  (a_1,\dots, a_m) \mapsto (a_1,-a_1,\dots, a_m,-a_m); \\
f' & \colon A^{m}\to A_0^{(2m+1)}, \  (a_1,\dots, a_m) \mapsto (a_1,-a_1,\dots, a_m,-a_m, 0).
\end{split}
\end{equation}

\begin{definition} \phantomsection \label{def:W}
	\begin{enumerate}[label=(\roman*)]
		\item If $n$ is odd, $n=2m-1$, we let $\iota \colon W\hookrightarrow K^{n}(A)$ be the strict transform of the image of $f\colon A^{m}\to A_0^{(2m)}$ under $\nu\colon K^n(A)\to A_0^{(2m)}$.
		\item If $n$ is even, $n=2m$, we let $\iota \colon W\hookrightarrow K^{n}(A)$ be the strict transform of the image of $f'\colon A^{m}\to A_0^{(2m+1)}$ under $\nu\colon K^n(A)\to A_0^{(2m+1)}$.
	\end{enumerate}
\end{definition}

Notice that the automorphism $-1$ of $A$ naturally acts on $K^{n}(A)$ for all $n$, as the natural action of $-1$ on $A^{[n+1]}$ stabilizes $K^{n}(A)$. It is known that $-1$ acts trivially on $H^2(K^n(A),\ZZ)$, see \cite{boissiere2011higher}. 
The map $f$ (resp. $f'$) factors through an embedding of $(A/\pm 1)^{(m)}$ into $A_0^{(2m)}$ (resp. into $A_0^{(2m+1)}$); a natural crepant resolution of $(A/\pm 1)^{(m)}$ is given by $\Km(A)^{[m]}$, the Hilbert scheme of points on the Kummer K3 surface $\Km(A)$.

\begin{lemma} \label{lem:KMO}
	For any $n\geq 2$, the subvariety $W\subset K^{n}(A)$ is the unique component of maximal dimension of the fixed locus of $-1$ acting on $K^{n}(A)$. Moreover, if $n=2m$ (resp. $n=2m-1$), then $f'$ (resp. $f$) induces an embedding $\iota\colon \Km(A)^{[m]}\hookrightarrow K^{2m}(A)$ (resp. $\iota\colon \Km(A)^{[m]} \hookrightarrow K^{2m-1}(A)$) with image $W$.
\end{lemma}
\begin{proof}
	The lemma is proven by Kamenova-Mongardi-Oblomkov \cite{KMO}. They show in \cite[Theorem 1.3]{KMO} that the fixed locus of $-1$ on $K^n(A)$ is the union of components which are varieties of $\mathrm{K}3^{[k]}$-type or points, and that there exists a unique component of maximal dimension $2m$ for $n=2m-1$ or $n=2m$. 
	As our $W\subset K^n(A)$ is evidently fixed by $-1$, it must be the component of maximal dimension of the fixed locus. 
	The local study done in \cite[Appendix B]{KMO} shows that the birational map $\Km(A)^{[m]}\dashrightarrow W$ induced by $f'$ if $n=2m$ or $f$ if $n=2m-1$ extends to an isomorphism.
\end{proof} 

The main result of this section is the following. 

\begin{proposition}\label{prop:computation}
	Let $\iota\colon W \hookrightarrow K^{n}(A)$ denote the embedding of Definition \ref{def:W}. Then the pull-back along $\iota$ induces a primitive embedding
	\[ \iota^*\colon H^2(K^{n}(A),\ZZ) (2) \hookrightarrow H^2(W,\ZZ) \]
	of lattices and Hodge structures, where on the left-hand side the form is multiplied by a factor $2$.
\end{proposition}
The cases $n=2$ and $n=3$ may be deduced from \cite{hassettTschinkel} and \cite{floccariHCKum3}, respectively.

\subsection{} \label{subsec:notationCohomology}
Let us fix some notation for the cohomology groups of the varieties we shall consider.
By \cite{beauville1983varietes}, for $n\geq 2$ there is a canonical primitive embedding of $H^2(A,\ZZ)$ into the second cohomology of $K^n(A)$. We have 
\begin{equation}\label{eq:cohomologyK^n}
H^2(K^n(A),\ZZ) = H^2(A,\ZZ)\oplus \ZZ\cdot \xi ,
\end{equation}
where $\xi$ is half the class of the divisor $E\subset K^n(A)$ parametrizing non-reduced subschemes, which is the exceptional divisor of the Hilbert-Chow resolution. The class~$\xi$ has square $-2(n+1)$ with respect to the Beauville-Bogomolov form; as a lattice, $H^2(K^n(A),\ZZ)$ is thus isometric to $\mathrm{U}^{\oplus 3} \oplus \langle -2(n+1)\rangle$. 

The Kummer surface $\Km(A)$ associated to $A$ contains $16$ exceptional curves $C_{\tau}$, parametrized by the $2$-torsion points $\tau\in A_2$. The pushforward along the rational map $\pi\colon A\dashrightarrow \Km(A)$ induces a primitive embedding of $H^2(A,\ZZ)(2)$ into $H^2(\Km(A),\ZZ)$ whose orthogonal is the Kummer lattice (\cite{Nikulin}), which is the saturation of the sublattice generated by the $16$ pairwise orthogonal $(-2)$-classes $[C_{\tau}]$. Hence, $H^2(\Km(A),\ZZ)$ contains with finite index the sublattice $H^2(A,\ZZ)(2) \oplus \langle [C_{\tau}]\rangle_{\tau\in A_2}$. 

For the cohomology of the Hilbert scheme $\Km(A)^{[m]}$, $m\geq 2$, we have by \cite{beauville1983varietes} a canonical primitive embedding of $H^2(\Km(A),\ZZ)$ into $H^2(\Km(A)^{[m]},\ZZ)$, and 
\begin{equation}\label{eq:cohomologyKm^m}
H^2(\Km(A)^{[m]},\ZZ) = H^2(\Km(A),\ZZ)\oplus \ZZ\cdot \delta,
\end{equation}
where $\delta$ is half the class of the divisor $D\subset \Km(A)^{[m]}$ of non-reduced subschemes; the Beauville-Bogomolov square of $\delta$ is $-2(m-1)$. Under the embedding of $H^2(\Km(A),\ZZ)$ into $H^2(\Km(A)^{[m]},\ZZ)$, the class $[C_{\tau}]$ of the exceptional curve over the node given by $\tau\in A_2$ corresponds to the class $[R_{\tau}]$ of the divisor $R_{\tau}$ of subschemes whose support intersects $C_{\tau}$.
Therefore, $H^2(\Km(A)^{[m]},\ZZ) $ contains with finite index a sublattice $H^2(A,\ZZ)(2) \oplus \langle [R_{\tau}]\rangle_{\tau\in A_2} \oplus \langle \delta\rangle$ and the saturation of $\langle [R_{\tau}]\rangle_{\tau\in A_2}$ is the Kummer lattice.

\subsection{} 
The next lemma is the first step towards Proposition \ref{prop:computation}. 
\begin{lemma} \label{lem:restrictionE}
	Consider the embedding $\iota\colon W\hookrightarrow K^n(A)$, for $n\geq 2$.
	\begin{enumerate}[label=(\roman*)]
	\item For $n=2$, the pull-back along $\iota\colon \Km(A) \hookrightarrow K^{2}(A)$ yields $$\iota^*(\xi) = [C_0] + \frac{1}{2} \sum_{\tau\in A_2} [C_{\tau}] \text{ in } H^2(\Km(A),\ZZ).$$ 
	\item For $n=2m-1$, the pull-back along $\iota\colon \Km(A)^{[m]}\hookrightarrow K^{2m-1}(A)$ yields
	$$\iota^*(\xi) = 2\delta + \frac{1}{2} \sum_{\tau\in A_2} [R_{\tau}] \text{ in } H^2(\Km(A)^{[m]} ,\ZZ).$$
	\item For $n=2m \geq 4$, the pull-back along $\iota\colon \Km(A)^{[m]}\hookrightarrow K^{2m}(A)$ yields
	$$\iota^*(\xi) = 2\delta + [R_0] + \frac{1}{2} \sum_{\tau\in A_2} [R_{\tau}] \text{ in } H^2(\Km(A)^{[m]},\ZZ).$$ 
 \end{enumerate}
\end{lemma}
Notice that $\tfrac{1}{2}\sum_{\tau} [C_{\tau}]$ is indeed an integral class, see \cite[Chapter 14, \S3.3]{huyK3}.
The proof of the lemma will be reduced to some local computations.

\begin{example}\label{example:basic}
	Consider the involution $-1\colon \CC^2\to \CC^2$. The quotient $\CC^2/\pm 1$ has a nodal singularity at the image $\overline{0}$ of the origin, which is resolved by the blow-up $\mathrm{Bl}_{\overline{0}}(\CC^2/\pm 1)\to \CC^2$. We denote by $C \subset \mathrm{Bl}_{\overline{0}}(\CC^2/\pm 1) $ the exceptional curve. 
	\begin{enumerate}[label=(\alph*)]
	\item The graph map $\hat{\phi}\colon \CC^2\to (\CC^2)^2$ given by $y\mapsto (y,-y)$ yields a Cartesian diagram 
	\begin{equation}
	\begin{tikzcd}
	\mathrm{Bl}_{\overline{0}}(\CC^2/\pm 1) \arrow[hook]{r}{\phi} \arrow{d}{\nu} & \CC^{[2]} \arrow{d}{\nu} \\
	(\CC^2/\pm 1) \arrow[hook]{r}{\bar{\phi}} & (\CC^2)^{(2)};
	\end{tikzcd}
	\end{equation}
	by \cite[Lemma 3.13]{KMO}, the image of $\phi$ is the fixed locus of $-1$ acting on $\CC^{[2]}$. Let $F\subset \CC^{[2]}$ be the divisor of non-reduced subschemes. Then, as divisors, $\phi^*(F)=C$; indeed, the image of $\phi$ and $F$ intersect transversely in the curve $\phi(C)$.
	\item Consider now the map $\hat{\phi}\colon \CC^2\to (\CC^2)^3$ given by $y\mapsto (y,-y,0)$. It induces the Cartesian commutative diagram 
	\begin{equation}\label{eq:diagrammino}
	\begin{tikzcd}
		\mathrm{Bl}_{\overline{0}}(\CC^2/\pm 1) \arrow[hook]{r}{\phi} \arrow{d}{\nu} & \CC^{[3]} \arrow{d}{\nu} \\
		(\CC^2/\pm 1) \arrow[hook]{r}{\bar{\phi}} & (\CC^2)^{(3)};
	\end{tikzcd}
	\end{equation}
	by \cite[Lemma 3.13]{KMO}, the fixed locus of $-1$ acting on $\CC^{[3]}$ consists of the image of $\phi$ and an isolated fixed point. Let $F\subset \CC^{[3]}$ be the divisor of non-reduced subschemes. Then we claim that, as divisors, we have $\phi^*(F)=3C$.
	
	To see this, let $B\subset (\CC^2)^{[3]}$ be the Brian\c{c}on variety at $0$, that is, the subvariety $B\coloneqq \nu^{-1}((0,0,0))$ that parametrizes subschemes fully supported at $0\in \CC^2$, and let $j\colon B\hookrightarrow  \CC^{[3]}$ denote the embedding. 
	The set-theoretic intersection of~$F$ and the image of $\phi$ is $\phi(C)$, which is contained in~$B$. 
	In our case, by \cite[IV.2]{Briancon77}, the Brian\c{c}on variety $B$ is isomorphic to the weighted projective plane $\PP(1,1,3)$; following \cite[\S4]{hassettTschinkel}, there exists a closed  embedding $\PP(1,1,3)\hookrightarrow \PP^4$ realizing $B$ as the cone over the twisted cubic curve and the restriction $j^*(\mathcal{O}_{(\CC^2)^{[3]}}(F))$ equals~$\mathcal{O}_B(-2H)$, where we denote by $\mathcal{O}_{B}(H)$ the line bundle which induces the embedding $B\hookrightarrow \PP^4$.
	Moreover, $-1\colon \CC^{[3]}\to \CC^{[3]}$ acts on $B$ with fixed locus the vertex of the cone and the twisted cubic curve at the base, which is thus identified with $\phi(C)$. 
	The embedding $\phi_{|_C}\colon C\hookrightarrow  (\CC^{2})^{[3]}$ factors as the composition of embeddings $C\xrightarrow{\phi'} B \xrightarrow{j}  (\CC^{2})^{[3]}$. Therefore,
	\begin{equation}\label{eq:degree}
		\deg\bigl(\mathcal{O}_{(\CC^2)^{[3]}}(F)_{|_{\phi(C)}}\bigr) = \deg \bigl(\mathcal{O}_B(-2H)_{|_{\phi'(C)}}\bigr) =-6,
	\end{equation}
	as $H$ is a hyperplane divisor in $\mathbb{P}^4$ and the twisted cubic curve has degree $3$. If $k$ denotes the multiplicity of $\phi^*(F)$ at $C$, we have $-2k= \deg\bigl(\mathcal{O}_{(\CC^2)^{[3]}}(F)_{|_{\phi(C)}}\bigr) = -6$, and we conclude that $k=3$. 
	
	To justify the last claim, consider a smooth and projective surface $Y$ with an involution $\theta$ with isolated fixed points. Each fixed point gives a node on the quotient $Y/\theta$; the blow-up of these nodes is a smooth and projective surface $X$. We choose a fixed point $p\in Y$, and let $C_p \subset X$ be the exceptional curve over the corresponding node. We then have the Cartesian diagram of projective varieties 
	\[
	\begin{tikzcd}
		X \arrow[hook]{r}{\phi} \arrow{d}{\nu} & Y^{[3]} \arrow{d}{\nu} \\
		Y/\theta \arrow[hook]{r}{\bar{\phi}} & Y^{(3)},
	\end{tikzcd}
	\]
	induced by the map $Y\to Y^{3}$ sending $y$ to $(y, \theta(y), p)$.	
	We can find an analytic neighborhood $U\subset Y$ of $p$ stable under $\theta$ which is identified with a neighborhood~$V$ of the origin in $\CC^2$ in such a way that the action of $\theta$ on $U$ corresponds to the action of~$-1$ on~$V$. Thus, denoting by $F$ the exceptional divisor of $Y^{[3]}\to Y$ and by $k$ the multiplicity of $\phi^*(F)$ at $C_p$, we have to show that $-2k = \deg\bigl(\mathcal{O}_{Y^{[3]}}(F)_{|_{\phi(C_p)}}\bigr)$. But $\phi^*(F)$ is supported over the exceptional curves $C_q$ of $X\to Y/\theta$, which are pairwise orthogonal $(-2)$-curves. Therefore, by the projection formula, we have 
	\[
	-2k=\int_X [C_p]\cdot [\phi^*(F)] =  \int_{Y^{[3]}}[\phi(C_p)]\cdot [F] = \deg\bigl(O_{Y^{[3]}}(F)_{|_{\phi(C_p)}}\bigr).
	\]	
\end{enumerate}
\end{example}
\begin{proof}[Proof of Lemma \ref{lem:restrictionE}]
	The exceptional divisor $E\subset K^n(A)$ is the restriction to $K^n(A)\subset A^{[n+1]}$ of the exceptional divisor of the Hilbert-Chow resolution $A^{[n+1]}\to A^{(n+1)}$, which we still denote by $E$.
	We consider the Cartesian diagram 
	\begin{equation}\label{eq:diagramma}
	\begin{tikzcd}
	\Km(A)^{[m]} \arrow[hook]{r}{\iota} \arrow{d} & A^{[n+1]} \arrow{d}{\nu}\\
	(A/\pm 1)^{(m)} \arrow[hook]{r}{\bar{\iota}} & A^{(n+1)},	
	\end{tikzcd}
	\end{equation}
	where $n=2m$ or $n=2m-1$: in the first case, $\bar{\iota}(\overline{a_1}, \dots, \overline{a_m}) = (a_1,-a_1, \dots, a_m,-a_m, 0)$, while in the second $\bar{\iota}(\overline{a_1}, \dots, \overline{a_m}) = (a_1,-a_1, \dots, a_m,-a_m)$. It is immediate to check that, if $m\geq 2$, the restriction $\iota^*(E)$ is supported over the union of the Hilbert-Chow divisor $D\subset \Km(A)^{[m]}$ and the $16$ divisors $R_{\tau}\subset \Km(A)^{[m]}$, for $\tau\in A_2$, while if $m=1$ then $\iota^*(E)$ is supported over the $16$ exceptional curves $C_{\tau}\subset \Km(A)$, $\tau\in A_2$, which we will also denote by $R_{\tau}$ to ease notation.
	We now have the following cases to consider. 
	
	\textbf{Case 1}: let $\tau\in A_2$ be a $2$-torsion point; if $n$ is even, assume further that $\tau\neq 0$. We claim that then $\iota^*(E)$ has multiplicity $1$ along the divisor $R_{\tau}$. 
	
	Indeed, the image in $(A/\pm 1)^{(m)}$ of a general subscheme $\zeta \in R_{\tau}$ is a point of the form $(\overline{\tau}, \overline{a_2},\dots, \overline{a_m})$ with the $\overline{a_i}$ pairwise distinct smooth points of $A/\pm 1$, while $\iota(\zeta)$ is a subscheme in $A^{[n+1]}$ supported over the point $(\tau, \tau, a_2, -a_2, \dots, a_m,-a_m)$ of $A^{(n+1)}$ if $n=2m-1$ is odd or over $(\tau, \tau, a_2, -a_2, \dots, a_m,-a_m, 0)$ if $n=2m$ is even. 
	Let $\mathcal{U}$ be an analytic neighborhood of $\iota(\zeta)$ in $A^{[n+1]}$ stable under $-1$, and let $\mathcal{V}$ be the neighborhood $\iota^{-1}(\mathcal{U})$ of $\zeta$ in $\Km(A)^{[m]}$. Assume first that $n= 2m-1$, for some $m\geq 2$. Choosing $\mathcal{U}$ small enough,	diagram \eqref{eq:diagramma} restricts to
	\[
	\begin{tikzcd}
		\mathcal{V}\cong \bigl(\mathrm{Bl}_{\overline{\tau}}(U_{\tau}/\pm 1) \times \prod_{j=2}^m U_j\bigr) \arrow[hook]{r}{\iota} \arrow{d} & \mathcal{U} \cong \bigl(U_{\tau}^{[2]} \times \prod_{j=2}^m(U_{j} \times U_{j}^{-})\bigr) \arrow{d}{\nu} \\
		(U_{\tau}/\pm 1) \times \prod_{j=2}^m U_j \arrow[hook]{r}{\bar{\iota}} & U_{\tau}^{(2)} \times \prod_{j=2}^m(U_{j} \times U_{j}^{-})
	\end{tikzcd}\] 
	  where: $U_{\tau}\subset A$ is a neighborhood of $\tau$ stable under $-1$; $U_{j}\subset A$ is a neighborhood of $a_j$ and $U^-_{j}\coloneqq -1(U_j)$, for $j=2,\dots,m$; the morphism $\nu$ is the product of the Hilbert-Chow resolution $U_{\tau}^{[2]}\to U_{\tau}^{(2)}$ with the identity of each factor $(U_j\times U^{-}_j)$; the morphism $\bar{\iota}$ is given by $\bar{\iota}(\overline{\alpha}, \alpha_2,\dots, \alpha_m)= ((\alpha, -\alpha), \alpha_2,-\alpha_2, \dots, \alpha_m,-\alpha_m)$.
	  If instead $n=2m$ is even, choosing $\mathcal{U}$ small enough diagram \eqref{eq:diagramma} restricts to
	  \[
	  \begin{tikzcd}
	  	\mathcal{V}\cong \bigl(\mathrm{Bl}_{\overline{\tau}}(U_{\tau}/\pm 1) \times \prod_{j=2}^m U_j\bigr) \arrow[hook]{r}{\iota} \arrow{d} & \mathcal{U} \cong \bigl(U_{\tau}^{[2]} \times \prod_{j=2}^m(U_{j} \times U_{j}^{-}) \times U_0\bigr) \arrow{d}{\nu} \\
	  (U_{\tau}/\pm 1) \times \prod_{j=2}^m U_j \arrow[hook]{r}{\bar{\iota}} & U_{\tau}^{(2)} \times \prod_{j=2}^m(U_{j} \times U_{j}^{-}) \times U_0
	  \end{tikzcd}
	  \]
	  where: $U_{\tau}\subset A$ is a neighborhood of $\tau$ stable under $-1$; $U_{j}\subset A$ is a neighborhood of $a_j$ and $U^-_{j}\coloneqq -1(U_j)$, for $j=2,\dots,m$; $U_0\subset A$ is a neighborhood of $0$; the morphism $\nu$ is the product of the Hilbert-Chow resolution $U_{\tau}^{[2]}\to U_{\tau}^{(2)}$ with the identity of each factor $(U_j\times U^{-}_j)$ and of $U_0$; we have $\bar{\iota}(\overline{\alpha}, \alpha_2,\dots, \alpha_m)= ((\alpha, -\alpha), \alpha_2,-\alpha_2, \dots, \alpha_m,-\alpha_m, 0)$.
	  
	  In both cases, we have $E_{|_{\mathcal{U}}}= \mathrm{pr}_{1}^*(F)$, where $\mathrm{pr}_1\colon \mathcal{U} \to U_{\tau}^{[2]}$ is the projection onto the first factor and $F\subset U_{\tau}^{[2]}$ is the divisor of non-reduced subschemes.
	  Moreover, the divisor ${R_{\tau}}_{|_{\mathcal{V}}}$ equals $\widetilde{\mathrm{pr}}_1^*(C)$, where $\widetilde{\mathrm{pr}}_1\colon \mathcal{V}\to \mathrm{Bl}_{\overline{\tau}}(U_{\tau} /\pm 1)$ is the projection and $C\subset \mathrm{Bl}_{\overline{\tau}}(U_{\tau} /\pm 1)$ is the exceptional curve.
		We thus have the commutative diagram 
	  \[
	  \begin{tikzcd}
	  	\mathcal{V} \arrow[hook]{r}{\iota} \arrow{d}{\widetilde{\mathrm{pr}}_1} & \mathcal{U} \arrow{d}{\mathrm{pr}_1} \\
	  	\mathrm{Bl}_{\overline{\tau}}(U_{\tau} /\pm 1) \arrow[hook]{r}{\phi} & U_{\tau}^{[2]}
  	  \end{tikzcd}
    \] 
    where $\phi$ coincides with the restriction of the map $\mathrm{Bl}_{\overline{0}}(\CC^2/\pm 1) \to (\CC^2)^{[2]}$ of Example~\ref{example:basic}.(a) in a neighborhood of the exceptional curve $C\subset \mathrm{Bl}_{\overline{0}}(\CC^2\pm 1)$. We deduce that, as divisors, we have $$\iota^*(E_{|_{\mathcal{U}}})= ( \iota\circ \mathrm{pr}_1)^*(F) = (\widetilde{\mathrm{pr}}_1^* (\phi^*(F)))= \widetilde{\mathrm{pr}}_1^*(C)= {R_{\tau}}_{|_{\mathcal{V}}}. $$
    Hence, $\iota^*(E)$ has multiplicity $1$ along the divisor $R_{\tau}\subset \Km(A)^{[m]}$.
    
    \textbf{Case 2:} let $n=2m$ be even and consider the divisor $R_{0}$ of $\Km(A)^{[m]}$. We claim that $\iota^*(E)$ has multiplicity $3$ along the divisor $R_{0}$.
    
    Indeed, the image in $(A/\pm 1)^{(m)}$ of a general subscheme $\zeta\in R_{0}$ is a point of the form $(\overline{0}, \overline{a_2},\dots, \overline{a_m})$ with the $\overline{a_i}$ pairwise distinct smooth points of $A/\pm 1$, while $\iota(\zeta)$ is a subscheme in $A^{[n+1]}$ supported over the point $(0, 0, a_2, -a_2, \dots, a_m,-a_m, 0)$ of $A^{(n+1)}$.
    Let $\mathcal{U}$ be an analytic neighborhood of $\iota(\zeta)$ in $A^{[n+1]}$ such that $-1(\mathcal{U})=\mathcal{U}$, and let $\mathcal{V}$ be the analytic neighborhood $\iota^{-1}(\mathcal{U})$ of $\zeta$ in $\Km(A)^{[m]}$.     
    Choosing $\mathcal{U}$ small enough, diagram \eqref{eq:diagramma} restricts to 
    \[
    \begin{tikzcd}
    	\mathcal{V}\cong \bigl(\mathrm{Bl}_{\overline{0}}(U_{0}/\pm 1) \times \prod_{j=2}^m U_j\bigr) \arrow[hook]{r}{\iota} \arrow{d} & \mathcal{U} \cong \bigl(U_{0}^{[3]} \times \prod_{j=2}^m(U_{j} \times U_{j}^{-}) \bigr) \arrow{d}{\nu} \\
    	(U_{0}/\pm 1) \times \prod_{j=2}^m U_j \arrow[hook]{r}{\bar{\iota}} & U_{0}^{(3)} \times \prod_{j=2}^m(U_{j} \times U_{j}^{-})
    \end{tikzcd}
    \]
    where: $U_0\subset A$ is a neighborhood of $0$ stable under $-1$; $U_j\subset A$ is a neighborhood of $a_j$ and $U^{-}_j\coloneqq -1(U_j)$, for $j=2,\dots, m$; the morphism $\nu$ is the product of the Hilbert-Chow resolution $U_0^{[3]}\to U_0^{(3)}$ with the identity of each factor $(U_j\times U^{-}_j)$; we have $\bar{\iota}(\overline{\alpha}, \alpha_2,\dots, \alpha_m)= ((\alpha, -\alpha, 0), \alpha_2,-\alpha_2, \dots, \alpha_m,-\alpha_m)$.
   	Then $E_{|_{\mathcal{U}}}= \mathrm{pr}_{1}^*(F)$, where $\mathrm{pr}_1\colon \mathcal{U} \to U_0^{[3]}$ is the projection and $F\subset U_0^{[3]}$ is the divisor of non-reduced subschemes.
    Moreover, ${R_{0}}_{|_{\mathcal{V}}}=\widetilde{\mathrm{pr}}_1^*(C)$, where $\widetilde{\mathrm{pr}}_1\colon \mathcal{V}\to \mathrm{Bl}_{\overline{0}}(U_{0} /\pm 1)$ is the projection and $C\subset \mathrm{Bl}_{\overline{0}}(U_{0} /\pm 1)$ is the exceptional curve.
    We have the commutative diagram 
    \[
    \begin{tikzcd}
    	\mathcal{V} \arrow[hook]{r}{\iota} \arrow{d}{\widetilde{\mathrm{pr}}_1} & \mathcal{U} \arrow{d}{\mathrm{pr}_1} \\
    	\mathrm{Bl}_{\overline{0}}(U_{0} /\pm 1) \arrow[hook]{r}{\phi} & U_{0}^{[3]}
    \end{tikzcd}
    \] 
    where $\phi$ coincides with the restriction of the map $\mathrm{Bl}_{\overline{0}} (\CC^2/\pm 1) \to (\CC^2)^{[3]}$ of Example~\ref{example:basic}.(b) in a neighborhood of $C\subset \mathrm{Bl}_{\overline{0}}(\CC^2/\pm 1)$, and we deduce that $$\iota^*(E_{|_{\mathcal{U}}})= ( \iota\circ \mathrm{pr}_1)^*(F) = (\widetilde{\mathrm{pr}}_1^* (\phi^*(F)))= \widetilde{\mathrm{pr}}_1^*(3C)= 3 {R_{0}}_{|_{\mathcal{V}}} $$
    as divisors.
    Hence, $\iota^*(E)$ has multiplicity $3$ along the divisor $R_{0}\subset \Km(A)^{[m]}$.
    
    \textbf{Case 3:} assume now that $m\geq 2$ and consider the Hilbert-Chow divisor $D\subset \Km(A)^{[m]}$. We claim that $\iota^*(E)$ has multiplicity $2$ along $D$. 
    
    In fact, the image in $(A/\pm 1)^{(m)}$ of a general subscheme $\zeta\in D$ is a point $(\overline{x}, \overline{x}, \overline{a_3},\dots, \overline{a}_m)$, for $m-1$ distinct smooth points $\overline{x}, \overline{a_3},\dots, \overline{a_m}$ of $(A/\pm 1)$. The subscheme $\iota(\zeta)$ in $A^{[n+1]}$ is supported over $(x, x, -x,-x, a_3,-a_3, \dots, a_m,-a_m)$ if $n=2m-1$ is odd, or over $(x,x,-x,-x, a_3,-a_3, \dots, a_m,-a_m, 0) $ if $n=2m$ is even.
    Let $\mathcal{U}\subset A^{[n+1]}$ be an analytic neighborhood of $\iota(\zeta)$, and let $\mathcal{V}\subset \Km(A)^{[m]}$ be its preimage $\iota^{-1}(\mathcal{U})$. 
    If $n=2m-1$ is odd, shrinking $\mathcal{U}$ if necessary, there exist neighborhoods $U_x $ of $x$ and $U_j$ of $a_j$ in $A$ for $j=3,\dots, m$, such that diagram \eqref{eq:diagramma} restricts to 
    \[
    \begin{tikzcd}
    	\mathcal{V}\cong \bigl(U_x^{[2]} \times \prod_{j=3}^m U_j\bigr) \arrow[hook]{r}{\iota} \arrow{d} & \mathcal{U} \cong \bigl(U_{x}^{[2]} \times (U_{x}^-)^{[2]} \times \prod_{j=3}^m(U_{j} \times U_{j}^{-})\bigr) \arrow{d}{\nu} \\
    	U_{x}^{(2)} \times \prod_{j=3}^m U_j \arrow[hook]{r}{\bar{\iota}} & U_{x}^{(2)}\times (U_x^-)^{(2)} \times \prod_{j=3}^m(U_{j} \times U_{j}^{-})
    \end{tikzcd}
	\] 
	where: $U^-_x$ and $U_{j}^-$ denote the image via $-1$ of $U_x$ and $U_j$ respectively, the map $\nu$ is the product of the Hilbert-Chow morphisms $U_x^{[2]}\to U_x^{(2)}$ and $(U_x^-)^{[2]}\to (U^-_x)^{(2)}$ with the identity of each factor $(U_j\times U^{-}_j)$ for $j=3,\dots,m $, and $\bar{\iota}$ is given by $\bar{\iota}((x_1,x_2), \alpha_3,\dots, \alpha_m)= ((x_1, x_2), (-x_1, -x_2), \alpha_3,-\alpha_3, \dots, \alpha_m,-\alpha_m)$.
    If instead $n=2m$ is even, there exist neighborhoods $U_{x}\subset A$ of $x$ and $U_j\subset A$ of $a_j$ for $j=3,\dots, m$, and a neighborhood $U_0\subset A$ of $0$ such that diagram \eqref{eq:diagramma} restricts to 
    \[
    \begin{tikzcd}
    	\mathcal{V}\cong \bigl(U_x^{[2]} \times \prod_{j=3}^m U_j\bigr) \arrow[hook]{r}{\iota} \arrow{d} & \mathcal{U} \cong \bigl(U_{x}^{[2]} \times (U_x^{-})^{[2]} \times \prod_{j=3}^m(U_{j} \times U_{j}^{-}) \times U_0 \bigr) \arrow{d}{\nu} \\
    	U_x^{(2)} \times \prod_{j=3}^m U_j \arrow[hook]{r}{\bar{\iota}} & U_{x}^{(2)}\times (U_x^{-})^{(2)} \times \prod_{j=3}^m(U_{j} \times U_{j}^{-}) \times U_0
    \end{tikzcd}
    \]
    where: $U^-_x$ and $U_{j}^-$ denote the image via $-1$ of $U_x$ and $U_j$ respectively, the map $\nu$ is the product of the Hilbert-Chow morphisms $U_x^{[2]}\to U_x^{(2)}$ and $(U_x^-)^{[2]}\to (U^-_x)^{(2)}$ with the identity of $U_0$ and of each factor $(U_j\times U^{-}_j)$ for $j=3,\dots,m $, and we have $\bar{\iota}((x_1,x_2), \alpha_3,\dots, \alpha_m)= ((x_1, x_2), (-x_1, -x_2), \alpha_3,-\alpha_3, \dots, \alpha_m,-\alpha_m, 0)$.
    
    In both cases, the divisor $E_{|_{\mathcal{U}}} $ breaks into two components $E_1$ and $E_2$ such that $E_1=\mathrm{pr}_1^*(F)$ and $E_2 = \mathrm{pr}_2^*(F^-)$, where $\mathrm{pr}_1\colon \mathcal{U}\to U_x^{[2]}$ and $\mathrm{pr}_2\colon \mathcal{U}\to (U_x^-)^{[2]}$ denote the projections and $F\subset U_{x}^{[2]}$ (resp. $F^-\subset (U^-_x)^{[2]}$) is the divisor of non-reduced subschemes. Moreover, $D_{|_{\mathcal{V}}} = \widetilde{\mathrm{pr}}_1^*(F)$, where $\widetilde{\mathrm{pr}}_1\colon \mathcal{V}\to U_x^{[2]}$ is the projection. Notice that $-1\in \Aut(A)$ induces an isomorphism $-1\colon U_x^{[2]}\to (U_x^{-})^{[2]} $ which identifies $F$ with~$F^{-}$.
    From the commutative diagram 
    \[
    \begin{tikzcd}
    	\mathcal{V} \arrow[hook]{rr}{\iota} \arrow{d}{\widetilde{\mathrm{pr}}_1} && \mathcal{U} \arrow{d}{\mathrm{pr}_1\times \mathrm{pr}_2}\\
    	U_{x}^{[2]} \arrow[hook]{rr}{\mathrm{id} \times (-1)} && U_x^{[2]} \times (U^{-}_x)^{[2]}
    \end{tikzcd}
	\]
	we deduce that $\iota^*(E_1) = D_{|_{\mathcal{V}}} = \iota^*(E_2)$, and therefore $\iota^*(E)$ has multiplicity $2$ at the divisor $D\subset \Km(A)^{[m]}$.
    
	\textbf{Conclusion.} Recalling that we have $[E]=2\xi$ in $H^2(K^n(A),\ZZ)$ and $[D]=2\delta$ in $H^2(\Km(A)^{[m]},\ZZ)$, the $3$ cases discussed above yield $(i), (ii)$ and $(iii)$.    
\end{proof}

\subsection{} The key property of the embedding $\iota\colon \Km(A)^{[m]}\to K^{n}(A)$ of Definition \ref{def:W} is that it deforms together with $K^{n}(A)$. 
More precisely, we have the following. 
\begin{proposition}\label{prop:deformation}
	Let $\mathcal{K}\to B$ be a smooth proper family of complex manifolds of $\mathrm{Kum}^{n}$-type over a connected manifold $B$, such that for some $0\in B$ we have $\mathcal{K}_0 = K^{n}(A)$, for some abelian surface $A$. Let $\iota_0\colon \Km(A)^{[m]}\to K^{n}(A)$ be the embedding discussed above, where $n=2m-1$ or $n=2m$. 
	Then, up to a finite \'etale base-change, there exist a smooth and proper family $\mathcal{W}\to B$ of manifolds of $\mathrm{K}3^{[m]}$-type with $\mathcal{W}_0 = \Km(A)^{[m]}$ and a closed embedding $\iota\colon \mathcal{W}\hookrightarrow \mathcal{K}$ over $B$ extending $\iota_0$.
\end{proposition}
\begin{proof} 
	By \cite[Theorem 2.1]{hassettTschinkel}, the group $\Aut_0(X)$ of automorphisms inducing the identity on $H^2(X,\ZZ)$ is deformation invariant, for any hyper-K\"ahler manifold $X$. It was computed in \cite{boissiere2011higher} that for $K^n(A)$ we have $\Aut_0(K^n(A))= A_{n+1} \rtimes \langle -1\rangle$; thus, any $K$ of $\mathrm{Kum}^n$-type admits an action of $\Aut_0(K)\cong (\ZZ/(n+1)\ZZ)^4 \rtimes \ZZ/2\ZZ$. 
	In particular, if $\mathcal{K}\to B$ is a family as in the statement, the action of $-1 \in \Aut(A)$ on the fibre $\mathcal{K}_0$ extends to a fibrewise automorphism of $\mathcal{K}\to B$, up to a finite \'etale base-change possibly needed in order to trivialize the monodromy action on $\Aut_0(\mathcal{K}_0)$. For each $b\in B$, the action of $-1$ on $\mathcal{K}_b$ is a symplectic involution. 
	By \cite{Camere}, each component $\mathcal{Y}$ of the fixed locus $\mathcal{K}^{(-1)}$ is a smooth and proper family $\mathcal{Y}\to B$ of symplectic manifolds, or just the image of a section of $\mathcal{K}\to B$. 
	By Lemma \ref{lem:KMO}, for $n=2m$ or $n=2m-1$ there exists a unique such component $\mathcal{W}$ of maximal relative dimension~$2m$, such that $\mathcal{W}_0 = \Km(A)^{[m]}$ is embedded into $ \mathcal{K}_0$ via $\iota_0\colon \Km(A)^{[m]}\hookrightarrow K^{n}(A)$. We thus have obtained the desired subfamily $\mathcal{W}\to B$ of manifolds of $\mathrm{K}3^{[m]}$-type.
\end{proof} 

As a consequence, the cohomology class $[\iota(\Km(A)^{[m]})]$ remains algebraic, and in particular a Hodge class, on any deformation of $K^{n}(A)$.

\begin{definition}
	A canonical Hodge class on a hyper-K\"ahler manifold $X$ is a Hodge class $\eta \in H^{2k}(X,\QQ)$ which remains Hodge on any deformation of~$X$.
\end{definition}
The most obvious examples are given by the Chern classes of $X$.
Any canonical Hodge class $\alpha\in H^{4k}(X,\QQ)$ on $X$ satisfies the generalized Fujiki relation \cite{fujiki1987}: there exists a constant $c(\alpha)\in \QQ$ such that for any $\gamma\in H^2(X,\QQ)$ we have
\begin{equation}\label{eq:genFujiki}
\int_{X} \alpha\cdot (\gamma)^{2n-2k} = \tfrac{(2n-2k)!}{(n-k)!\, 2^{n-k}} \cdot c(\alpha)\cdot q_X(\gamma,\gamma)^{n-k},
\end{equation}
where $\dim X =2n$ and $q_X$ is the Beauville-Bogomolov form on $H^2(X,\QQ)$. The constant~$c(1)$ is called the reduced Fujiki constant of $X$.

Let now $\mathcal{K}\to B$ be a family of manifolds of $\mathrm{Kum}^{n}$-type as in Proposition \ref{prop:deformation}. Up to a finite \'etale base-change, that proposition yields a subfamily $\mathcal{W}\to B$ of manifolds of $\mathrm{K}3^{[m]}$-type, where $n=2m$ or $n=2m-1$. We let $\iota\colon\mathcal{W}\to \mathcal{K}$ denote the embedding.
\begin{lemma}\label{lem:restriction}
	For any $b\in B$, the pull-back $\iota_b^*\colon H^2(\mathcal{K}_b,\ZZ)\to H^2(\mathcal{W}_{b},\ZZ)$ is injective and multiplies the form by a factor $2$.
\end{lemma}
\begin{proof}
	Extending our family $\mathcal{K}\to B$ if necessary, we may assume that the very general fibre has Picard rank $0$, so that $H^2(\mathcal{K}_b,\ZZ)$ is an irreducible Hodge structure for very general $b$. A symplectic form on $\mathcal{K}_b$ restricts to a symplectic form on $\mathcal{W}_b$ (see \cite[Proposition 3]{Camere}), and hence $\iota^*_b$ is not zero. But then $\iota^*_b\colon H^2(\mathcal{K}_b,\QQ)\to H^2(\mathcal{W}_b,\QQ)$ must be an embedding of rational Hodge structures for very general $b\in B$. Since $H^2(\mathcal{K}_b,\ZZ)$ is torsion-free, we conclude that $\iota^*_b\colon H^2(\mathcal{K}_b,\ZZ) \to H^2(\mathcal{W}_b,\ZZ)$ is injective, for all $b\in B$. 
	
	For the second assertion, it suffices to consider $\iota\colon \Km(A)^{[m]}\to K^{n}(A)$. Assume first that $n=2$ and $m=1$. The cohomology class $[\iota(\Km(A))]\in H^4(K^2(A),\ZZ)$ is a canonical Hodge class and, hence, there exists a constant $c$ such that 
	\[ q_{\Km(A)} (\iota^*(\gamma), \iota^*(\gamma))= \int_{\Km(A)} \iota^*(\gamma)^{2} = \int_{K^2(A)} [\Km(A)]\cdot \gamma^2=c \cdot q_{K^{2}(A)} (\gamma, \gamma), \]
	for any $\gamma\in H^2(K^2(A),\ZZ)$. For the class $\xi\in H^2(K^2(A),\ZZ)$, we have $q_{K^2(A)}(\xi,\xi)=-6$ while $q_{\Km(A)}(\iota^*(\xi), \iota^*(\xi))= -12$ by Lemma \ref{lem:restrictionE}.$(i)$, and we get $c=2$.
	
	Consider next the case when $n=2m-1$ is odd, with $m\geq 2$. 
	The cohomology class $[\iota(\Km(A)^{[m]})]\in H^{4m-4}(K^{2m-1}(A),\ZZ)$ is a canonical Hodge class. Using the generalized Fujiki relation \eqref{eq:genFujiki} for $[\iota(\Km(A)^{[m]})]$ and Fujiki relation for $\mathrm{Km}(A)^{[m]}$, there exists a constant $c$ such that
	\begin{equation}\label{eq:genFujiki1}
		\begin{split}
		\tfrac{(2m)!}{m! 2^m} \cdot
		c_{\mathrm{K}3^{[m]}}\cdot q_{\Km(A)^{[m]}}(\iota^*(\gamma),\iota^*(\gamma))^m & = \int_{\Km(A)^{[m]}} \iota^*(\gamma)^{2m} \\
		&  = \tfrac{(2m)!}{m! 2^m} \cdot 
		 c \cdot q_{K^{2m-1}(A)} (\gamma, \gamma)^m,
		 \end{split}
	\end{equation} 
for any $\gamma\in H^{2}(K^{2m-1}(A),\ZZ)$, where $c_{\mathrm{K}3^{[m]}} = 1$ is the reduced Fujiki constant of manifolds of $\mathrm{K}3^{[m]}$-type (computed in \cite{beauville1983varietes}).
For the class $\xi\in H^2(K^{2m-1}(A),\ZZ)$ which is half the class of the divisor of non-reduced subschemes, we have $q_{K^{2m-1}(A)} (\xi, \xi)= -4m$, and, by Lemma \ref{lem:restrictionE}.$(ii)$, we have 
$q_{\mathrm{Km}(A)^{[m]}} ( \iota^*(\xi),\iota^*(\xi)) = -8m$. We thus obtain
\[ 
 (-8m)^m = c \cdot (-4m)^m,
\]
which gives $c=2^m$. Therefore, for any $\gamma\in H^2(K^{2m-1}(A),\ZZ)$ equation \eqref{eq:genFujiki1} yields  $q_{\Km(A)^{[m]}}(\iota^*(\gamma),\iota^*(\gamma))^m =2^m \cdot q_{K^{2m-1}(A)} (\gamma, \gamma)^m$, and, hence, $\iota^*$ multiplies the form by~$2$ (for $m$ even, the sign is determined as $\iota^*$ sends K\"ahler classes to K\"ahler classes). 

Finally, assume $n=2m\geq 4$ is even. As above, there exists a constant $c$ such that 
	\begin{align}\label{eq:genFujiki2}
	c_{\mathrm{K}3^{[m]}}\cdot q_{\Km(A)^{[m]}}(\iota^*(\gamma),\iota^*(\gamma))^m & = \int_{\Km(A)^{[m]}} \iota^*(\gamma)^{2m} 
	= c \cdot q_{K^{2m}(A)} (\gamma, \gamma)^m,
\end{align} 
for all $\gamma\in H^2(K^{2m}(A),\ZZ)$. For the class $\xi\in H^2(K^{2m}(A),\ZZ)$ which is half the class of non-reduced subschemes, we have $q_{K^{2m}(A)} (\xi, \xi)= -4m-2$, while we have $q_{\Km(A)^{[m]}} (\iota^*(\xi),\iota^*(\xi)) = -8m -4 $ by Lemma \ref{lem:restrictionE}.$(iii)$. As $c_{\mathrm{K}3^{[m]}}=1$, we obtain 
\[ 
(-8m-4)^m = c \cdot (-4m-2)^m,
 \]
and, hence, $c=2^m$; by \eqref{eq:genFujiki2}, we have  $q_{\Km(A)^{[m]}}(\iota^*(\gamma),\iota^*(\gamma))^m =2^m \cdot q_{K^{2m}(A)} (\gamma, \gamma)^m$ for any $\gamma\in H^2(K^{2m}(A),\ZZ)$, and again we deduce that $\iota^*$ multiplies the form by $2$. 
\end{proof}

\subsection{} 
We can finally conclude our study of the pull-back $\iota^*$.
\begin{proof}[Proof of Proposition \ref{prop:computation}]
	Thanks to Lemma \ref{lem:restriction}, it remains to show that the image of $\iota^*\colon H^2(K^{n}(A),\ZZ)\hookrightarrow H^2(\Km(A)^{[m]},\ZZ)$ is a saturated sublattice. As in \eqref{eq:cohomologyK^n}, we have $H^2(K^n(A),\ZZ) = H^2(A,\ZZ)\oplus \langle \xi\rangle$, and the Hodge structure on $H^2(K^n(A),\ZZ)$ is determined by that on $H^2(A,\ZZ)$ via this equality.
	Moreover, the Hodge structure on $H^2(\Km(A)^{[m]},\ZZ)$ is determined by that on its primitive sublattice $H^2(A,\ZZ)(2)$ (\S\ref{subsec:notationCohomology}). Deforming to a very general complex torus $A$ such that $H^2(A,\ZZ)$ is an irreducible Hodge structure, we see that the restriction $\iota^*$ necessarily maps the summand $H^2(A,\ZZ)$ of $H^2(K^n(A),\ZZ)$ to the sublattice $H^2(A,\ZZ)(2)$ of $H^2(\Km(A)^{[m]},\ZZ)$. 
	In fact, the restriction of $\iota^*$ to $H^2(A,\ZZ)\subset H^2(K^n(A),\ZZ)$ is surjective onto $H^2(A,\ZZ)(2)\subset H^2(\Km(A)^{[m]},\ZZ)$, as follows from the discriminant-index formula (see \cite[Chapter 14]{huyK3}) since $\iota^*$ is injective and multiplies the form by $2$ by Lemma \ref{lem:restriction}.
	
	Assume first that $n=2$ and $m=1$. Then, by the above and Lemma \ref{lem:restrictionE}, the image of $\iota^*$ is the sublattice of $H^2(\Km(A),\ZZ)$ generated by $H^2(A,\ZZ)(2)$ and the class $[C_0]+\tfrac{1}{2}\sum_{\tau\in A_2}[C_{\tau}]$. This sublattice is saturated: if $w\coloneqq \alpha u + \beta \cdot \bigl([C_0] +\tfrac{1}{2} \sum_{\tau\in A_2} [C_{\tau}]\bigr)$ is an integral class in $H^2(\Km(A),\ZZ)$, for some primitive $u\in H^2(A,\ZZ)(2)$ and some rational numbers $\alpha$ and~$\beta$, then from $q_{\Km(A)}(w, [C_{\tau}])=-\beta$ for any $\tau\neq 0$ we deduce that $\alpha$ and $\beta$ are integers, and therefore $w$ belongs to $\mathrm{im}(\iota^*)$. 
	 
	For $m\geq 2$, the lattice $H^2(\Km(A)^{[m]},\ZZ)$ contains with finite index the sublattice $H^2(A,\ZZ)(2)\oplus \langle [R_{\tau}]\rangle_{\tau\in A_2}\oplus \langle \delta\rangle$. 
	By Lemma \ref{lem:restrictionE}, if $n=2m-1$, the image of $\iota^*$ is the sublattice of $H^2(\Km(A)^{[m]},\ZZ)$ generated by $H^2(A,\ZZ)(2)$ and the class $2\delta + \tfrac{1}{2} \sum_{\tau\in A_2 }[R_{\tau}]$, while if $n=2m$ the image of $\iota^*$ is the sublattice of $H^2(\Km(A)^{[m]},\ZZ)$ generated by $H^2(A,\ZZ)(2)$ and the class $2\delta+[R_0] +\tfrac{1}{2} \sum_{\tau\in A_2}[R_{\tau}]$. In both cases, the same argument used above shows that the image of $\iota^*$ is a saturated sublattice of $H^2(\Km(A)^{[m]},\ZZ)$. 
\end{proof}

\begin{remark}
	A manifold of $\mathrm{Kum}^n$-type $K$ admits at least $(n+1)^4$ involutions acting trivially on the second cohomology, whose fixed loci are union of $\mathrm{K}3^{[j]}$-manifolds by~\cite{KMO}. Our results adapt immediately to all components of maximal dimension. 
	For a lower dimensional component $\iota\colon Z\hookrightarrow K$, the argument of Lemma \ref{lem:restriction} gives that $\iota^*\colon H^2(K,\ZZ)\to H^2(Z,\ZZ)$ is injective and multiplies the form by some positive constant~$k$; however, determining this constant and the saturation of the image of $\iota^*$ becomes more difficult.
\end{remark}

\section{K3 surfaces associated with varieties of generalized Kummer type}

\subsection{} In this section we shall conclude the proof of our main results. 
\begin{definition}\label{def:S_K}
	Let $K$ be a variety of $\mathrm{Kum}^{n}$-type. The associated $\mathrm{K}3$ surface $S_K$ is the $\mathrm{K}3$ surface with transcendental lattice Hodge isometric to $H^2_{\mathrm{tr}}(K,\ZZ)(2)$.
\end{definition}
In the special case of the generalized Kummer variety $K^n(A)$ on an abelian surface~$A$, we have $H^2_{\mathrm{tr}}(K^n(A),\ZZ)=H^2_{\mathrm{tr}}(A,\ZZ)$, and the K3 surface $S_{K^n(A)}$ is isomorphic to the Kummer K3 surface $\Km(A)$. This motivates our Definition~\ref{def:S_K}. Note that for $K$ very general $S_K$ is projective of Picard rank $16$, and thus it is not a Kummer K3 surface.

\begin{lemma}
	The $\mathrm{K}3$ surface $S_K$ exists and it is unique up to isomorphism.
\end{lemma}
\begin{proof}
	The lattice $H^2_{\mathrm{tr}}(K,\ZZ)(2)$ has rank at most $6$ and signature~$(2,k)$. Any such lattice can be embedded primitively in the K3 lattice and, by the Torelli theorem, there exists a K3 surface $S_K$ whose transcendental lattice $H^2_{\mathrm{tr}}(S_K,\ZZ)$ is Hodge isometric to $H^2_{\mathrm{tr}}(K,\ZZ)(2)$ (see \cite[Corollary 2.10]{Morrison1984}). 
	The K3 surface $S_K$ is necessarily projective, of Picard rank at least $16$, and it is determined up to isomorphism because projective K3 surfaces with Hodge isometric transcendental lattices and Picard rank at least~$12$ are isomorphic (see \cite[Chapter 16, Corollary 3.8]{huyK3}). 
\end{proof}

The K3 surfaces associated to varieties of generalized Kummer type can be easily characterized via the Torelli theorem.
\begin{lemma}
	Let $S$ be a projective $\mathrm{K}3$ surface. Then $S$ is isomorphic to the $\mathrm{K}3$ surface $S_K$ associated to some variety $K$ of $\mathrm{Kum}^{n}$-type if and only if there exists a primitive embedding of lattices $H^2_{\mathrm{tr}}(S,\ZZ)\hookrightarrow \mathrm{U}(2)^{\oplus 3} \oplus \langle -4n-4\rangle$.
\end{lemma}
\begin{proof}
	The lattice $\mathrm{U}(2)^{\oplus 3} \oplus \langle -4n-4 \rangle$ is the Beauville-Bogomolov lattice $\Lambda_{\mathrm{Kum}^{n}}(2)$ of $\mathrm{Kum}^{n}$-manifolds with the form multiplied by $2$. Thus, by definition, for the K3 surface $S_K$ associated to a variety $K$ of $\mathrm{Kum}^n$-type we have a primitive embedding of~$H^2_{\mathrm{tr}}(S_K,\ZZ) $ into the lattice $\mathrm{U}(2)^{\oplus 3} \oplus \langle -4n-4\rangle$.
	Conversely, assume that $S$ is a projective K3 surface and $l\colon H^2_{\mathrm{tr}}(S,\ZZ)\hookrightarrow \Lambda_{\mathrm{Kum}^{n}}(2)$ is a primitive embedding. Let $T\subset \Lambda_{\mathrm{Kum}^{n}}$ be the primitive sublattice such that $\mathrm{im}(l)=T(2)$, and equip $T$ with the Hodge structure induced by that on $H^2_{\mathrm{tr}}(S,\ZZ)$. By the surjectivity of the period map for hyper-K\"ahler manifolds (\cite[Theorem 8.1]{Huy99}), there exists a manifold~$K$ of $\mathrm{Kum}^n$-type with $H^2_{\mathrm{tr}}(K,\ZZ)$ Hodge isometric to $T$. By construction, we have a Hodge isometry 
	\[ H^2_{\mathrm{tr}}(S,\ZZ) \xrightarrow{ \ \sim \ } H^2_{\mathrm{tr}}(K,\ZZ)(2);
	\] 
	${K}$ is projective by \cite[Theorem 3.11]{Huy99}, since its transcendental lattice has signature~$(2,k)$. Moreover, the K3 surface $S$ is isomorphic to $S_K$, as they have Hodge isometric transcendental lattices and Picard rank at least~$16$.
\end{proof}

Passing to rational coefficients, a similar argument yields the following.
\begin{lemma}\label{lem:isogenies}
	Let $S$ be a projective $\mathrm{K}3$ surface. The following are equivalent:
	\begin{enumerate}[label=(\roman*)]
		\item there exists a rational Hodge isometry $H^2_{\mathrm{tr}}(S,\QQ)\xrightarrow{\ \sim \ } H^2_{\mathrm{tr}}(S_K,\QQ)$, where $S_K$ is the $\mathrm{K}3$ surface associated to a variety $K$ of $\mathrm{Kum}^n$-type, for some $n\geq 2$; 
		\item there exists an isometric embedding of $H^2_{\mathrm{tr}}(S,\QQ)$ into $\bigl(\mathrm{U}^{\oplus 3}\oplus \langle -m \rangle\bigr)\otimes_{\ZZ} \QQ$, for some positive integer $m$. 
	\end{enumerate}
\end{lemma}
\begin{proof}
	For any non-zero rational number $k$, we have $\mathrm{U}(k)\otimes_{\ZZ} \QQ \cong \mathrm{U}\otimes_{\ZZ} \QQ $. If~$(i)$ holds, then $H^2_{\mathrm{tr}}(S,\QQ)$ is Hodge isometric to $H^2_{\mathrm{tr}}(K,\QQ)(2)$, and, hence, it embeds isometrically into $\bigl(\mathrm{U}(2)^{\oplus 3} \oplus \langle -4n-4\rangle\bigr)\otimes_{\ZZ} \QQ \cong \bigl(\mathrm{U}^{\oplus 3} \oplus \langle -n-1\rangle\bigr)\otimes_{\ZZ}\QQ$. 
	Conversely, if~$S$ is a projective K3 surface with an embedding $H^2_{\mathrm{tr}}(S,\QQ) \hookrightarrow \bigl(\mathrm{U}^{\oplus 3} \oplus \langle -m\rangle\bigr)\otimes_{\ZZ} \QQ$ for some positive $m\in\ZZ$, we find as well an isometric embedding $$j\colon H^2_{\mathrm{tr}}(S,\QQ)\hookrightarrow \bigl(\mathrm{U}(2)^{\oplus 3} \oplus \langle -4n-4\rangle\bigr) \otimes_{\ZZ} \QQ,$$ for some $n\geq 2$. The quadratic space on the right-hand side is $\Lambda_{\mathrm{Kum}^n}(2)\otimes_{\ZZ}\QQ$. We choose a primitive sublattice $T\subset \Lambda_{\mathrm{Kum}^{n}}$ such that $\mathrm{im}(j)$ coincides with $T(2)\otimes_{\ZZ} \QQ$; under this identification, the Hodge structure on $H^2_{\mathrm{tr}}(S,\QQ)$ induces a Hodge structure of K3-type on $T$. By the surjectivity of the period map, there exists a projective variety $K$ of $\mathrm{Kum}^{n}$-type such that $H^2_{\mathrm{tr}}(K,\ZZ)$ is Hodge isometric to $T$, equipped with this Hodge structure. Hence, by construction, $H^2_{\mathrm{tr}}(S,\QQ)$ is Hodge isometric to $H^2_{\mathrm{tr}}(S_K,\QQ)$.
\end{proof}

\subsection{}\label{subsec:construction}
A priori, the K3 surface $S_K$ and the variety $K$ of $\mathrm{Kum}^n$-type are only related at the level of lattices and Hodge structures. But in fact they are related geometrically, as follows.  
Choose a family $\mathcal{K}\to B$ of manifolds of $\mathrm{Kum}^{n}$-type, with one fibre $\mathcal{K}_{b_0}= K$ and another fibre $\mathcal{K}_{b_1}=K^{n}(A)$ isomorphic to the generalized Kummer variety on an abelian surface $A$. 
By Proposition \ref{prop:deformation}, up to a finite \'etale base-change, we obtain a family $\mathcal{W}\to B$ of $\mathrm{K}3^{[m]}$-manifolds with an embedding $\iota\colon \mathcal{W}\to \mathcal{K}$ over $B$, where $n=2m-1$ or $n=2m$, which extends the embedding $\iota_{b_1}\colon \Km(A)^{[m]}\hookrightarrow K^{n}(A)$ of Lemma \ref{lem:KMO}.
Taking the fibre over $b_0$, we get a submanifold $\iota \colon W_K \hookrightarrow K $ of $\mathrm{K}3^{[m]}$-type.
\begin{lemma}\label{lem:transcendentalIsometry}
	The pull-back along $\iota\colon W_K\hookrightarrow K$ induces a primitive embedding of lattices and Hodge structures
	\[ \iota^*\colon H^2(K,\ZZ)(2) \hookrightarrow H^2(W_K,\ZZ).\]
	In particular, $\iota^*$ induces a Hodge isometry $H^2_{\mathrm{tr}}(K,\ZZ)(2)\xrightarrow{\ \sim \ } H^2_{\mathrm{tr}}(W_K,\ZZ)$ of transcendental lattices.
\end{lemma} 
\begin{proof}
	The lemma follows immediately from Proposition \ref{prop:computation} as $\iota\colon W_K\hookrightarrow K$ deforms to the embedding $\Km(A)^{[m]}\hookrightarrow K^{n}(A)$ there considered.
\end{proof}
 We refer to \cite{huybrechts2010geometry, Bri08, BM14b} for background on moduli spaces of sheaves on K3 surfaces.

\begin{proposition}\label{prop:moduliSpace}
	Let $S_K$ be the $\mathrm{K}3$ surface associated to the $\mathrm{Kum}^{n}$-variety $K$. Then: 
	\begin{enumerate}[label=(\roman*)]
		\item if $n=2$, $W_K$ is isomorphic to the $\mathrm{K}3$ surface $S_K$;
		\item if $n\geq 3$, $W_K\subset K$ is birational to a smooth and projective moduli space $M_{S_K, H}(v)$, for some primitive Mukai vector $v$ and a $v$-generic polarization $H$.
	\end{enumerate}
\end{proposition}
\begin{proof}
	By Lemma \ref{lem:transcendentalIsometry} and the definition of $S_K$, there exists a Hodge isometry $H^2_{\mathrm{tr}}(W_K,\ZZ)\xrightarrow{ \ \sim \ } H^2_{\mathrm{tr}}(S_K,\ZZ)$ of transcendental lattices. If $n=2$, $W_K$ and $S_K$ are isomorphic as they are K3 surfaces with Hodge isometric transcendental lattice and large Picard rank. If $n\geq 3$, the proposition follows applying \cite[Proposition 4]{Addington2016}.
\end{proof}
\begin{remark}
	In fact, for $n\geq 3$, $W_K$ is isomorphic to a moduli space of Bridgeland stable objects on $S_K$, by \cite[Thm 1.2 (c)]{BM14a} (see also \cite[Proposition 2.3]{mongardiwandel}).
\end{remark}

\subsection{} Let $S$ be a projective K3 surface, and let $M=M_{S,H}(v)$ be a smooth and projective moduli space of stable sheaves of dimension $\geq 4$, for a primitive Mukai vector $v$ and a $v$-generic polarization $H$.
Then, following \cite{mukai1987moduli, O'G97}, there exists a quasi-universal sheaf $\mathcal{U}$ over $S\times M$ of similitude $\rho$, and the map 
$$\alpha\mapsto \frac{1}{\rho}\cdot \mathrm{pr}_{M,*}\bigl[ \mathrm{ch}(\mathcal{U})^{\vee} \cdot \sqrt{\mathrm{td}(S)} \cdot \mathrm{pr}_S^*(\alpha)\bigr]_3$$
induces a Hodge isometry 
$\theta\colon H^2_{\mathrm{tr}}(S,\ZZ)\xrightarrow{\ \sim \ } H^2_{\mathrm{tr}}(M,\ZZ)$ of transcendental lattices. Here, the notation $[-]_3$ indicates that we take the K\"unneth component in $H^6(S\times M,\ZZ)$ and $\mathrm{pr}_M$, $\mathrm{pr}_S$ denote the projections from $S\times M$ onto the two factors. 
The standard conjectures hold for both the surface $S$ and the $\mathrm{K}3^{[n]}$-variety $M$, as we recalled in \S\ref{subsec:standardConj}.
Hence, they hold for $S\times M$, and it follows that $\theta$ is the correspondence induced by the algebraic cycle $\bigl[\frac{1}{\rho}\mathrm{ch}(\mathcal{U})^{\vee} \cdot \sqrt{\mathrm{td}(S)}\bigr]_2 $ on $S\times M$.
 
\begin{remark}\label{rmk:h^2}
In fact, the K\"unneth components $\mathsf{h}^2(S)$ and $\mathsf{h}^2(M)$ as well as their transcendental parts $\mathsf{h}^2_{\mathrm{tr}}(S)$ and $\mathsf{h}^2_{\mathrm{tr}}(M)$ are well-defined homological motives; the Hodge isometry~$\theta$ is the realization of an isomorphism of motives $
\Theta\colon \mathsf{h}^2_{\mathrm{tr}} (S) \to \mathsf{h}^2_{\mathrm{tr}}(M) $ in $\mathsf{Mot}$. This again uses the standard conjectures for $S\times M$ and in particular the conservativity of the realization functor when restricted to $\langle \h(S\times M)\rangle_{\mathsf{Mot}}$ (Remark \ref{rmk:conservative}).
\end{remark}

\subsection{} 
Let $K$ be a variety of $\mathrm{Kum}^{n}$-type. By the results of Foster \cite{foster} recalled in Theorem~\ref{thm:foster}, the K\"unneth component $\h^2(K)$ of the motive of $K$ is well-defined in $\mathsf{Mot}$ and $\langle \h^2(K)\rangle_{\mathsf{Mot}}$ is a semisimple and abelian neutral Tannakian category; the transcendental part $\h^2_{\mathrm{tr}}(K)$ of $\h^2(K)$ is also a well-defined homological motive. The next result implies Theorem~\ref{thm:1.1} from the introduction.
\begin{theorem}\label{thm:half2}
	Let $K$ be a variety of $\mathrm{Kum}^{n}$-type and let $S_K$ be the $\mathrm{K}3$ surface associated to $K$. There exists an isomorphism 
	\[ \h^2_{\mathrm{tr}}(S_K) \xrightarrow{ \ \sim \ } \h^2_{\mathrm{tr}}(K) \]
	of motives, whose realization is a Hodge isometry 
	$H^2_{\mathrm{tr}}(S_K,\QQ) \xrightarrow{\ \sim \ } H^2_{\mathrm{tr}} (K,\QQ)(2)$. 
\end{theorem}
\begin{proof}
	For $n=2m$ or $n=2m-1$, we find as in \S\ref{subsec:construction} a $\mathrm{K}3^{[m]}$-variety $W_K$ which is isomorphic to $S_K$ for $n=2$ and birational to a moduli space of stable sheaves on $S_K$ if $n\geq 3$, and that admits an embedding $\iota\colon W_K \hookrightarrow K$ such that the restriction induces a Hodge isometry $\iota^*\colon H_{\mathrm{tr}}^2(K,\ZZ)(2)\xrightarrow{\ \sim \ } H^2_{\mathrm{tr}}(W_K,\ZZ)$. 
	
	As $\iota^*$ is clearly induced by an algebraic cycle, it is the realization of a morphism $t\colon \h^2_{\mathrm{tr}}(K) \to \h^2_{\mathrm{tr}}(W_K)$ of motives. But in fact $t$ is a morphism in the subcategory $\langle \h^2_{\mathrm{tr}}(K) \oplus \h^2_{\mathrm{tr}}(W_K)\rangle_{\mathsf{Mot}}$, which is an abelian and semisimple neutral Tannakian category as the standard conjectures hold for $W_K$ and $\h^2(K)$. Since the realization of $t$ is an isomorphism of rational Hodge structures, it follows that $t$ is an isomorphism of motives; the realization of its inverse $t^{-1}$ is a Hodge isometry $H^2_{\mathrm{tr}}(W_K,\QQ)\xrightarrow{ \ \sim \ } H^2_{\mathrm{tr}}(K,\QQ)(2)$. This concludes the proof for $n=2$, since $S_K$ is isomorphic to $W_K$ in this case.
	
	If $n\geq 3$, let $\phi\colon M_{S_K, H}(v)\dashrightarrow W_K$ be a birational map, where $M_{S_K,H}(v)$ is a smooth and projective moduli space of stable sheaves on $S_K$. Then $\phi$ induces a morphism $\Phi_*\colon \h^2(M_{S_K,H}(v))\to \h^2(W_K)$ of motives whose realization is a Hodge isometry $\phi_*\colon H^2(M_{S_K,H}(v),\ZZ)\xrightarrow{\ \sim \ } H^2(W_K,\ZZ)$, by \cite[Corollary 5.2]{Huy99}. As the standard conjectures hold for $\mathrm{K}3^{[m]}$-varieties, $\Phi_*$ is an isomorphism of motives. From Remark \ref{rmk:h^2} we obtain an isomorphism of motives $\Theta \colon \h^2_{\mathrm{tr}}(S_K) \xrightarrow{ \ \sim \ } \h^2_{\mathrm{tr}}(M_{S_K, H}(v))$, whose realization is a Hodge isometry $H^2_{\mathrm{tr}}(S_K,\ZZ)\xrightarrow{\ \sim \ } H^2_{\mathrm{tr}}(M_{S_K,H}(v),\ZZ)$.
	We conclude that the composition 
	$t^{-1} \circ\Phi_*\circ \Theta\colon \h^2_{\mathrm{tr}}(S_K) \to \h^2_{\mathrm{tr}}(K)$ is an isomorphism of motives, whose realization is the desired Hodge isometry $H^2_{\mathrm{tr}}(S_K,\QQ)\xrightarrow{ \ \sim \ } H^2_{\mathrm{tr}}(K,\QQ)(2)$.
\end{proof}

Thanks to the works of O'Grady, Markman, Voisin, Varesco, recalled in the introduction, Theorem \ref{thm:half2} yields the following result, which was stated as Theorem \ref{thm:2}. 
\begin{theorem} \label{thm:associatedK3}
	Let $K$ be a variety of $\mathrm{Kum}^n$-type with associated $\mathrm{K}3$ surface $S_K$. Then:
	\begin{enumerate}[label=(\roman*)]
		\item the Kuga-Satake correspondence for $S_K$ is algebraic;
		\item the Hodge conjecture holds for any power of $S_K$.
	\end{enumerate}
\end{theorem}
\begin{proof}
	By Theorem \ref{thm:half2} there exists an algebraic cycle inducing a Hodge isometry 
	\[
	H^2_{\mathrm{tr}}(S_K,\QQ)\xrightarrow{\ \sim \ } H^2_{\mathrm{tr}}(K,\QQ)(2). 
	\]
    The Kuga-Satake Hodge conjecture (Conjecture \ref{conj:KSHC}) holds for $K$ by \cite[Theorem 0.5]{voisinfootnotes}. By Remark~\ref{rmk:functoriality}, the same conclusion holds for $S_K$, and we obtain $(i)$. 
	
	The K3 surfaces $S_K$ come in countably many $4$-dimensional families of general Picard rank $16$, corresponding to the polarized families of varieties of generalized Kummer type. By construction (see \S\ref{subsec:functoriality}) the Kuga-Satake variety of $S_K$ is isogenous to a power of $\mathrm{KS}(K)$, and, hence, to a power of a Weil fourfold with discriminant~$1$ by \cite{markman2019monodromy, O'G21}. As Conjecture \ref{conj:KSHC} holds for any of the K3 surfaces~$S_K$, we can apply \cite[Theorem 0.2]{varesco} to deduce that the Hodge conjecture holds for all powers of the K3 surfaces $S_K$.
\end{proof}

\subsection{} We finally complete the proof of our results.
\begin{theorem}\label{thm:K3surfaces}
	Let $S$ be a projective $\mathrm{K}3$ surface such that there exists an isometric embedding $H^2_{\mathrm{tr}}(S,\QQ) \hookrightarrow \bigl(\mathrm{U}^{\oplus 3} \oplus \langle -m\rangle\bigr) \otimes_{\ZZ} \QQ$ for some positive integer~$m$. Then:
	\begin{enumerate}[label=(\roman*)]
		\item the Kuga-Satake correspondence for $S$ is algebraic;
		\item the Hodge conjecture holds for any power of $S$.
	\end{enumerate}
\end{theorem} 
\begin{proof}
	Let $S$ be as in the statement.
	By Lemma \ref{lem:isogenies}, for some $n\geq 2$ there exist a variety $K$ of $\mathrm{Kum}^n$-type and a rational Hodge isometry $\phi\colon H^2_{\mathrm{tr}}(S,\QQ) \xrightarrow{\ \sim \ } H^2_{\mathrm{tr}}(S_K,\QQ)$. 
	By \cite{Buskin, huybrechtsMotives}, $\phi$ is algebraic and the motives of $S$ and $S_K$ are isomorphic.
	Therefore, by Remark~\ref{rmk:functoriality} and Theorem \ref{thm:associatedK3}.$(i)$, the Kuga-Satake correspondence for~$S$ is algebraic, and, by Theorem \ref{thm:associatedK3}.$(ii)$, the Hodge conjecture holds for any power of~$S$. 
\end{proof}	

We next prove Theorem \ref{thm:1}, which we state again below.

\begin{theorem}\label{thm:main'}
	Let $A$ be an abelian fourfold of Weil type with discriminant $1$. Then the Hodge conjecture holds for $A$ and any of its powers.
\end{theorem}
\begin{proof}
	By \cite{O'G21} and \cite{markman2019monodromy}, there exists a $\mathrm{Kum}^n$-variety $K$ such that $A$ is isogenous to the intermediate Jacobian $J^3(K)$; moreover, the Kuga-Satake variety of $K$ is isogenous to a power of $A$.
	The Kuga-Satake variety of the K3 surface $S_K$ associated to $K$ is thus also isogenous to a power of $A$. By Theorem \ref{thm:associatedK3}, the Kuga-Satake Hodge conjecture holds for $S_K$ and the Hodge conjecture holds for any power of $S_K$. Therefore, Theorem~\ref{prop:KSHC} implies that the Hodge conjecture holds for~$A$ and all of its powers.	
\end{proof}

We conclude with the following corollary, which follows from a result of Varesco \cite{varesco2023hodge}.
\begin{corollary} \label{cor:corollary'}
	Let $S$ be a projective K3 surface and let $K$ be a projective variety of $\mathrm{Kum}^n$-type. Assume that $\phi\colon H^2_{\mathrm{tr}}(S_K,\QQ) \xrightarrow{\ \sim \ } H^2_{\mathrm{tr}}(K,\QQ)(k)$ is a Hodge isometry, where on the right-hand side the form is multiplied by some positive $k\in \QQ$. Then $\phi$ is induced by an algebraic cycle on $S\times K$.
\end{corollary}
\begin{proof}
	As $\Lambda_{\mathrm{Kum}^n} = \mathrm{U}^{\oplus 3}\oplus \langle -2n-2\rangle$ and $\mathrm{U}(k)\otimes_{\ZZ} \QQ \cong \mathrm{U}\otimes_{\ZZ}\QQ$ for any $0\neq k \in \QQ$, there exists by hypothesis an isometric embedding 
	$$H^2_{\mathrm{tr}}(S_K,\QQ)\hookrightarrow \mathrm{U}^{\oplus 3}_{\QQ}\oplus \langle -m\rangle_{\QQ},$$
	for some positive integer $m$.
	Hence, by Theorem \ref{thm:2}, the Kuga-Satake Hodge conjecture holds for $S$. Moreover, this conjecture holds for $K$ by \cite[Theorem~0.5]{voisinfootnotes}, and the Lefschetz standard conjecture in degree $2$ for $K$ is proven in \cite[Theorem~1.1]{foster}. 
 	We can therefore apply \cite[Corollary~4.6]{varesco2023hodge} to obtain that $\phi$ is algebraic.
	\end{proof}

\bibliographystyle{smfplain}
\bibliography{bibliographyNoURL}{}

\end{document}